\documentclass[MSNbibl,nameyear,dvips]{arxstspdf}
\usepackage{graphicx}
\usepackage{url,breakurl}
\usepackage{flushend}
\usepackage{stfloats}

\volume{29}
\issue{4}
\pubyear{2014}
\firstpage{619}
\lastpage{639}
\doi{10.1214/14-STS483} 
\docsubty{FLA}

\makeatletter

\newcommand{\lleft}{\left}
\newcommand{\rrvert}{\vert}
\newcommand{\rright}{\right}
\newcommand{\llvert}{\vert}
\renewcommand{\cite}[1]{\citet{#1}}
\newcommand{\eqref}[1]{(\ref{#1})}
\newcommand{\bb}{\boldsymbol{\beta}}
\newcommand{\bg}{\mathbf{g}}
\newcommand{\bt}{\beta}
\newcommand{\E}{\mathsf{E}}
\newcommand{\ga}{g}
\newcommand{\half}{\frac{1}{2}}
\newcommand{\iid}{i.i.d.\ }
\newcommand{\ind}{\mathbf{1}}
\newcommand{\ve}{\varepsilon}
\renewcommand{\th}{\theta}

\newtheorem{proposition}{Proposition}[section]
\newtheorem{lemma}[proposition]{Lemma}
\newtheorem{theorem}[proposition]{Theorem}
\newtheorem{corollary}[proposition]{Corollary}
\makeatother

\begin{document}
\begin{frontmatter}

\title{The Bayesian Analysis of Complex, High-Dimensional Models: Can~It~Be~CODA?}
\runtitle{CODA Bayes?}

\begin{aug}
\author[A]{\fnms{Y.}~\snm{Ritov}\corref{}\ead[label=e1]{yaacov.ritov@gmail.com}\ead[label=u1,url]{http://pluto.mscc.huji.ac.il/\textasciitilde yaacov}},
\author[B]{\fnms{P. J.}~\snm{Bickel}\ead[label=e2]{bickel@stat.berkeley.edu}\ead[label=u2,url]{http://www.stat.berkeley.edu/\textasciitilde bickel}},
\author[C]{\fnms{A. C.}~\snm{Gamst}\ead[label=e3]{acgamst@math.ucsd.edu}\ead[label=u3,url]{http://biostat.ucsd.edu/acgamst.htm}}
\and
\author[D]{\fnms{B. J. K.}~\snm{Kleijn}\ead[label=e4]{B.Kleijn@uva.nl}\ead[label=u4,url]{http://staff.science.uva.nl/\textasciitilde bkleijn/}}
\runauthor{Ritov, Bickel, Gamst and Kleijn}

\affiliation{The Hebrew University, University of California, Berkeley,
University of California, San Diego, Korteweg-de~Vries Institute for
Mathematics}

\address[A]{Y. Ritov is Professor, Department of Statistics,
The Hebrew University,
91905 Jerusalem, Israel \printead{e1,u1}.}
\address[B]{P. J. Bickel is Professor, Department of Statistics,
University of California,
Berkeley, California 94720-3860, USA \printead{e2,u2}.}
\address[C]{A. C. Gamst is Professor, Biostatistics and
Bioinformatics,
University of California,
San Diego, California 92093-0717, USA \printead{e3,u3}.}
\address[D]{B. J. K. Kleijn is Assistant Professor in Stochastics,
Korteweg-de Vries Institute for Mathematics,
P.O. Box 94248, 1090 GE Amsterdam,
The Netherlands \printead{e4,u4}.}
\end{aug}

%
\begin{abstract}
We consider the Bayesian analysis of a few complex, high-dimensional
models and show that intuitive priors, which are not tailored to the
fine details of the model and the estimated parameters, produce
estimators which perform poorly in situations in which good,
simple frequentist estimators exist.
The models we consider are: stratified sampling, the partial linear
model, linear and quadratic functionals of white noise and estimation
with stopping times. We present a strong version of Doob's consistency
theorem which demonstrates that the existence of a uniformly
$\sqrt{n}$-consistent estimator ensures that the Bayes
posterior is $\sqrt{n}$-consistent for values of the parameter
in subsets of prior probability 1. We also demonstrate that it
is, at least, in principle, possible to construct Bayes priors
giving both global and local minimax rates, using a suitable
combination of loss functions. We argue that there is no
contradiction in these apparently conflicting findings.
\end{abstract}

%
\end{frontmatter}

\section{Introduction}\label{introduction}

We show, through a number of illustrative examples of general phenomena,
some of the difficulties faced by application of the Bayesian paradigm
in the analysis of data from complex, high-dimensional models. We do
not argue against the use of Bayesian methods. However, we judge the
success of these methods from the frequentist/robustness point of view,
in the tradition of Bernstein, von Mises, and Le Cam; and more
recently \cite{cox-1993}. Some references are: \cite{bayarri-berger-2004},
\cite{diaconis-freedman-1993}, \cite{diaconis-freedman-1998},
\cite{freedman-1963}, \cite{freedman-1999}, \cite{lecam-yang-1990}
and \cite{lehmann-casella-1998}.

The extent to which the subjective aspect of data analysis is
central to the modern Bayesian point of view is debatable. See the
dialog between \cite{goldstein-2006} and \cite{berger-2006-a} and
the discussion of these two papers. However, central to any Bayesian
approach is the posterior distribution and the choice of prior. Even
those who try to reconcile Bayesian and frequentist approaches
(cf. \citeauthor{bayarri-berger-2004}, \citeyear{bayarri-berger-2004}), in the case of conflict, tend to
give greater weight to considerations based on the posterior
distribution, than to those based on frequentist assessments;
cf. \cite{berger-2006-b}.

An older and by now less commonly held point of view is that rational
inquiry requires the choice of a Bayes prior and exclusive use of the
resulting posterior in inference; cf. \cite{savage-1961}
and \cite{lindley-1953}. A~modern weaker version claims: ``Bayes
theorem provides a powerful, flexible tool for examining the actual
or potential ranges of uncertainty which arise when one or more
individuals seek to interpret a given set of data in light of
their own assumptions and `uncertainties about their
uncertainties','' (\citeauthor{smith-1986}, \citeyear{smith-1986}).
This point of view, which is the philosophical foundation of
the Bayesian paradigm, has consequences. Among them are the strong
likelihood principle, which says that all of the information in the
data is contained in the likelihood function, and the stopping time
principle, which says that stopping rules are irrelevant to
inference. We argue
that a commitment to these principles can easily lead to absurdities
which are striking in high dimensions. We see this as an argument
against ideologues.

We discuss our examples with these two types of Bayesian analysts in mind:
\begin{enumerate}[II.]
\item[I.]The Bayesian who views his prior entirely as reflecting his
beliefs and the posterior as measuring the changes in these beliefs due
to the data. Note that this implies strict adherence to the likelihood
principle, a uniform plug-in principle, and the stopping time principle.
Loss functions are not specifically considered in selecting the prior.
\item[II.]The pragmatic Bayesian who views the prior as a way of
generating decision theoretic procedures, but is content with priors
which depend on the data, insisting only that analysis starts with a
prior and ends with a posterior.
\end{enumerate}
For convenience, we refer to these Bayesians as type I and type II.

The main difference we perceive between the type II Bayesian and
a frequentist is that when faced with a specific problem, the type
II Bayesian selects a unique prior, uses Bayes rule to produce the
posterior and is then committed to using that posterior for all
further inferences. In particular, the type II Bayesian is free
to consider a particular loss function in selecting his prior and,
to the extent that this is equivalent to using a data-dependent
prior, change the likelihood; see \cite{wasserman-2000}. That
the loss function and prior are strongly connected has been
discussed by Rubin; see \cite{bock-2004}.

We show that, in high-dimensional (non or semiparametric) situations
Bayesian procedures based on priors chosen by one set of criteria, for
instance, reference priors, selected so that the posterior for a \mbox{possibly}
infinite dimensional parameter $\bt$ converges at the minimax rate, can
fail badly on other sets of criteria, in \mbox{particular}, in yielding
asymptotically minimax, semiparametrically efficient, or even
$\sqrt{n}$-consistent estimates for specific one-dimensional
parameters, $\th$. We show by example that priors leading to efficient
estimates of one-dimensional parameters can be constructed but that the
construction can be subtle, and typically does not readily also give
optimal global minimax rates for infinite dimensional features of the
model. It is true, as we argue in Section~\ref{positive}, that by
general considerations, Bayes priors giving minimax rates of convergence
for the posterior distributions for both single or ``small'' sets of
parameters and optimal rates in global metrics can be constructed,
in principle. Although it was shown in \cite{bickel-ritov-2003} that this
can be done consistently with the ``plug-in principle,'' the procedures
optimal for the composite loss are not natural or optimal, in general,
for either component. There is no general algorithm for constructing
such priors and we illustrate the failure of classical type~II Bayesian
extensions (see below) such as the introduction of hyperparameters.
Of course, Bayesian procedures are optimal on their own terms and we prove
an extension of a theorem of Doob at the end of this paper which makes
this point. As usual, the exceptional sets of measure zero in this
theorem can be quite large in nonparametric settings.

For smooth, low-dimensional parametric models, the Bernstein--von Mises
theorem ensures that for priors with continuous positive densities, all
Bayesian procedures agree with each other and with efficient frequentist
methods, asymptotically, to order $n^{-1/2}$; see, for example,
\cite{lecam-yang-1990}. At the other extreme, even with independent
and identically distributed data, little can be said about the extreme
nonparametric model $\mathcal{P}$, in which nothing at all is \mbox{assumed} about
the common distribution of the observations,~$P$. The natural quantities
to estimate, in this situation, are bounded linear functionals of the form
$\th= \int\ga(x) \,dP(x)$, with $\ga$ bounded and continuous.
There are unbiased, efficient estimates of these functionals and
Dirichlet process priors, concentrating on small but dense subsets of
$\mathcal{P}$ yielding estimates equivalent to order $n^{-1/2}$ to the
unbiased ones; see \cite{ferguson-1973}, for instance.

The interesting phenomena occur in models between these two extremes.
To be able to even specify natural unbounded linear functionals such as
the density $p$ at a point, we need to put smoothness restrictions on $P$
and, to make rate of convergence statements, global metrics such as $L_2$
must be used. Both Bayesians and frequentists must specify not only the
structural features of the model but smoothness constraints. Some
of our examples will show the effect of various smoothness
assumptions on Bayesian inference.

For ease of exposition, in each of our examples, we consider only
independent and identically distributed (i.i.d.) data and our focus is on
asymptotics and estimation.
Although our calculations are given almost exclusively for
specific Bayesian decision theoretic procedures under $L_2$-type
loss, we believe (but do not \mbox{argue} in detail) that the difficulties
we highlight carry over to other inference procedures, such as
the construction of confidence regions. Here is one implication
of such a result. Suppose that we can construct a Bayes credible
region $C$ for an infinite dimensional parameter $\bt$ which has
good frequentist and Bayesian properties, for example, asymptotic minimax
behavior for the specified model, as well as
$P (\bt\in C | X )$ and
$P (\bt\in C(X) | \bt ) > 1-\alpha$.
Then we automatically have a credible region $q(C)$ for any $q(\bt)$.
Our examples will show, however, that this region can be absurdly
large. So, while a Bayesian might argue that parameter estimation
is less important than the construction of credible regions, our
examples carry over to this problem as well.

Our examples will be discussed heuristically rather than
exhaustively, but we will make it clear when a formal
proof is needed. There is a body of theory in the area
(cf. \citeauthor{ghosal-etal-2000}, \citeyear{ghosal-etal-2000},
\citeauthor{kleijn-vandervaart-2006}, \citeyear{kleijn-vandervaart-2006},
and Bickel and Kleijn, \citeyear{bickel-kleijn-2012}, among others), giving
specific conditions under which some finite dimensional
intuition persists in higher dimensions. However, in this
paper we emphasize how easily these conditions are violated
and the dramatic consequences of such violations. Our examples
can be thought of as points of the parameter space to which the
prior we use assigns zero mass. Since all points of the
parameter space are similarly assigned zero mass, we have to
leave it to the readers to judge whether these points are,
in any sense, exceptional.

In Section~\ref{stratified}, we review an example introduced
in \cite{robins-ritov-1997}. The problem is that of
estimating a real parameter in the presence of an \mbox{infinite}
dimensional ``nuisance'' parameter. The parameter of interest
admits a very simple frequentist estimator which is
$\sqrt{n}$-consistent without any assumptions on the nuisance
parameters at all, as long as the sampling scheme is reasonable.
In this problem, the type I Bayesian is unable
to estimate the parameter of interest at the $\sqrt{n}$-rate at
all, without making severe smoothness assumptions on the infinite
dimensional nuisance parameter.
In fact, we show that if the nuisance parameters are too rough,
a type I Bayesian is unable to find any prior giving even a
consistent estimate of the parameter of interest.
On the other hand, we do construct
priors, tailored to the parameter we are trying to estimate, which
essentially reproduce the frequentist estimate.
Such priors may be satisfactory to a type II Bayesian, but surely
not to Bayesians of type I. The difficulty here is that a
commitment to the strong likelihood principle forces the
Bayesian analyst to ignore information about a parameter
which factors out of the likelihood and he is forced to
find some way of connecting that parameter to the parameter
of interest, either through reparameterization, which only
works if the nuisance parameter is smooth enough, or by
tailoring the prior to the parameter of interest.

In Section~\ref{partial}, we turn to the classical partial linear
regression model. We recall results of \cite{wang-etal-2011} which
give simple necessary and sufficient conditions on the nonparametric
part of the model for the parametric part to be estimated efficiently.
We use this example to show that a natural class of Bayes priors, which
yield minimax estimates of the nonparametric part of the model under
the conditions given in \cite{wang-etal-2011}, lead to Bayesian estimators
of the parametric part which are inefficient. In this case, there is
auxiliary information in the form of a conditional expectation which
factors out of the likelihood but is strongly associated with the
amount of information in the data about the parameter of interest.
The frequentist can estimate this effect directly, but the type I
Bayesian is forced to ignore this information and, depending on
smoothness assumptions, may not be able to produce a consistent
estimate of the parameter of interest at all. The fact that, for
a sieve-based frequentist approach, two different bandwidths are
needed for local and global estimation of parameters in this
problem has been known for some time; see \cite{chen-shiau-1994}.

In Section~\ref{plugin}, we consider the Gaussian white noise model
of \cite{ibragimov-hasminskii-1984}, \cite{donoho-johnstone-1994},
and \cite{donoho-johnstone-1995}. Here, we show that from a frequentist
point of view we can easily construct uniformly $\sqrt{n}$-consistent
estimates of all bounded linear functionals. However, both the type I
and type II Bayesian, who are restricted to the use of one and only
one prior, must fail to estimate some bounded linear functionals at
the $\sqrt{n}$-rate. This is because both are committed to the
plug-in principle and, as we argue, any plug-in estimator will fail to
be uniformly consistent. On the positive side, we show that it is easy
to construct tailor-made Bayesian procedures for any of the specific
functionals we consider in this section. Again, reparameterization,
which in this case is a change of basis, is important. The resulting
Bayesian procedures are capable of simultaneously estimating both the
infinite dimensional features of the model at the minimax rate and the
finite dimensional parameters of interest efficiently, but linear
functionals which might be of interest in subsequent inferences,
and cannot be estimated consistently, remain. We give a graphic example,
in this section, to demonstrate our claims.

A second example, examined in Section~\ref{norm}, concerns the
estimation of the norm of a high-dimen\-sional vector of means, $\bb$.
Again, for a suitably large set of $\bb$, we can show that the priors
normally used for minimax estimation of the vector of means in the
$L_2$ norm do not lead to Bayesian estimators of the norm of $\bb$ which
are $\sqrt{n}$-consistent. Yet there are simple frequentist estimators
of this parameter which are efficient. We then give a constructive
argument showing how a type II Bayesian can bypass the difficulties
presented by this model at the cost of selecting a nonintuitive prior
and various inconsistencies. A type II Bayesian can use a data-dependent
prior which allows for simultaneous estimation of $\bb$ at the minimax
rate and this specific parameter of interest efficiently. These
examples show that, in many cases, the choice of prior is subtle,
even in the type II context, and the effort involved in constructing
such a prior seems unnecessary, given that good, general-purpose
frequentist estimators are easy to construct for the same parameters.

In Section~\ref{stopping}, we give a striking example in which,
for Gaussian data with a high-dimensional parameter space, we can,
given any prior, construct a stopping time such that the Bayesian,
who must ignore the nature of the stopping times, estimates the
vector of means with substantial bias. This is a common feature
of all our examples. In high dimensions, even for large sample
sizes, the bias induced by the Bayes prior overwhelms the data.

In Section~\ref{positive}, we extend Doob's theorem, showing that
if a suitably uniform $\sqrt{n}$-consistent
estimator of a parameter exists, then necessarily the Bayesian estimator
of the parameter is $\sqrt{n}$-consistent on a set of parameter values
which has prior probability one. We also give another elementary result
showing that it is in principle possible to construct Bayes priors
giving both global and local minimax rates, using a suitable combination
of loss functions. We summarize our findings in Section~\ref{summary}.

In  Appendix~\ref{proofs}, we give proofs of many of the assertions we have
made in the previous sections. Throughout this paper, $\th$ is a
finite-dimensional parameter of interest, $\bt$ is an infinite-dimensional
nuisance parameter and $\ga$ is an infinite-dimensional parameter which is
important for estimating $\th$ efficiently, but is missing from the joint
likelihood for $(\th, \bt)$; $\ga$ might describe the sampling scheme,
the loss function or the specific functional $\th(\bt) = \th(\bt,\ga)$
of interest. We use $\pi$ for priors and $\ga$ and $\bt$ are given as
$\bg$ and $\bb$ when it is easier to think of them as infinite-dimensional
vectors than functions.

\section{Stratified Random Sampling}\label{stratified}

\cite{robins-ritov-1997} consider an infinite-dimen\-sional
model of continuously stratified random sampling in which one
has \iid observations $W_i=(X_i,R_i,\allowbreak  Z_i)$, $i=1, \ldots, n$;
the $X_i$ are uniformly distributed in $[0,1]^d$; and
$Z_i=R_i Y_i$. The variables $R_i$ and $Y_i$ are conditionally
independent given $X_i$ and take values in the set $\{0, 1\}$.
The function $\ga(X)=\E(R|X)$ is known, with $\ga> 0$ almost
everywhere, and $\bt(X)=\E(Y|X)$ is unknown. The parameter of
interest is $\th=\E(Y)$.

It is relatively easy to construct a reasonable estimator for
$\th$ in this problem. Indeed, the classical Horvitz--Thompson
(HT) estimator (cf. \citeauthor{cochran-1977}, \citeyear{cochran-1977}),
\[
\widehat{\th}= n^{-1}\sum_{i=1}^n
Z_i/\ga(X_i),
\]
solves the problem nicely. Because
\begin{eqnarray*}
\E\bigl\{RY/\ga(X)\bigr\} &=& \E \bigl\{\E(R|X)\E(Y|X)/\ga(X) \bigr\}
\\
&=& \E\E(Y|X) = \th,
\end{eqnarray*}
the estimator is consistent without any further assumptions.
If we assume that $\ga$ is bounded from below, the estimator
is $\sqrt{n}$-consistent and asymptotically normal.

\subsection{Type I Bayesian Analysis}\label{stratified:type-one}

As $\ga$ is known and we have assumed that the $X_i$ are uniformly distributed,
the only parameter which remains is $\bt$, where $\bt(X) = \E(Y|X)$.
Let $\pi$
be a prior density for $\bt$ with respect to some measure $\mu$. The joint
density of $\bt$ and the observations $W_1,\ldots, W_n$ is given by
\begin{eqnarray*}
p(\bt, \mathbf{W}) &=& \pi(\bt) \prod_{i : R_i=1}
\bt(X_i)^{Y_i} \bigl(1-\bt(X_i)
\bigr)^{1-Y_i}
\\
&&{} \cdot\prod_{i=1}^n
\ga(X_i)^{R_i} \bigl(1-\ga(X_i)
\bigr)^{1-R_i},
\end{eqnarray*}
as $Z_i = Y_i$ when $R_i = 1$. But this means that the posterior for
$\bt$
has a density $\pi(\bt|\mathbf{W})$ with
%
\begin{equation}
\label{CODApost} \pi(\bt|\mathbf{W}) \propto \pi(\bt) \prod
_{i : R_i=1} \bt(X_i)^{Y_i} \bigl(1-
\bt(X_i) \bigr)^{1-Y_i}.\hspace*{-20pt}
\end{equation}
Of course, this is a function of only those observations for which $R_i
= 1$,
that is, for which the $Y_i$ are directly observed. The observations
for which
$R_i = 0$ are deemed uninformative.

If $\beta$ is assumed to range over a smooth parametric model,
and the known $\ga$ is bounded away from $0$, one can check that the
Bernstein--von Mises theorem applies, and that the Bayesian estimator of
$\th$ is efficient, $\sqrt{n}$-consistent and necessarily better than
the HT estimator. Heuristically, this continues to hold for minimax
estimation of $\th$ and $\bt$ over ``small'' nonparametric models
for $\bt$; that is, sets of very smooth $\bt$;
see \cite{bickel-kleijn-2012}.

In the nonparametric case, if we assume that the prior for $\bt$ does
not depend on $\ga$, then, because the likelihood function does not depend
on $\ga$, the type I Bayesian will use the same procedure whether $\ga$ is
known or unknown; see \eqref{CODApost}. That is, the type I Bayesian
will behave as if $\ga$ were unknown. This is problematic because,
as \cite{robins-ritov-1997} argued and we now
show, unless $\bt$ or $\ga$ are sufficiently smooth, the type I Bayesian
cannot produce a consistent estimator of $\th$. To the best of our knowledge,
the fact that there is no consistent estimator of $\th$ when $\ga$ is unknown,
unless $\bt$ or $\ga$ are sufficiently smooth, has not been emphasized before.

Note that our assumption that the prior for $\bt$ does not depend on
$\ga$
is quite plausible. Consider, for example, an in-depth survey of students,
concerning their scholastic interests. The design of the experiment is
based on all the information the university has about the students.
However, the statistician is interested only in whether a student is firstborn
or not. At first, he gets only the list of sampled students with their
covariates. At this stage, he specifies his prior for $\bt$. If he is
now given $\ga$, there is no reason for him to change what he believes
about $\bt$, and no reason for him to include information about $\ga$
in his prior.

The fact that, if $\ga$ is unknown, $\th$ cannot be estimated
unless either $\ga$ or $\bt$ is smooth enough, is true even in
the one-dimensional case. Our analysis is similar to that in
\cite{robins-etal-2009}. Suppose the $X_i$ are uniformly
distributed on the unit interval, and $\ga$ is given by
\[
\ga(x) = \frac{1}{2} + \frac{1}{4}\sum_{i=0}^{m-1}
s_i \psi (mx - {i} ),
\]
where $m=m_n$ is such that $m_n/n\rightarrow\infty$;
the sequence $s_1,\ldots,s_m\in\{-1,1\}$
is assumed to be exchangeable with $\sum s_i = 0$, and
$\psi(x)=\ind (0\leq x <\half )-\ind (\half\leq x < 1 )$.
Furthermore, assume that $\bt(x)\equiv5/8$ or $\bt(x)\equiv\ga(x)$.
With probability converging to $1$, there will be no interval of length
$1/m$ with more than one $X_i$. However, given that there is one
$X_i\in(j/m,(j+1)/m)$, then the distribution of $(R_i,Z_i)$
is the same whether $\bt(x)\equiv5/8$ or $\bt(x)\equiv\ga(x)$, and
hence $\th$ is not identifiable; it can be either $5/8$ or $1/2$.
This completes the proof.

Note that, in principle, both $\E(YR|X)=\bt(X)\ga(X)$ and $E(R|X)=\ga(X)$
are, in general, estimable, but not uniformly to adequate precision on
``rough'' sets of $(\ga,\bt)$. One can also reparameterize in terms of
$\xi(X) = \E(YR|X)$ and $\th$. This forces $\ga$ into the likelihood,
but one still needs to assume $\xi(X)$ is very smooth. In the above
argument, the roughness of the model goes up with the sample size,
and this is what prevents consistent estimation.

\subsection{Bayesian Procedures with Good Frequentist Behavior}\label{stratified:good}

In this section, we study plausible priors for Type II Bayesian
inference. These
priors are related to those in \cite{wasserman-2004}, \cite{harmeling-toussaint-2007}
and \cite{li-2010}. We need to build knowledge of $\ga$ into the prior,
as we argued
in Section~\ref{stratified:type-one}. We do so first by following the
suggestion in
\cite{harmeling-toussaint-2007} for Gaussian models.

Following \cite{wasserman-2004}, we consider now a somewhat simplified version
of the continuously stratified random sampling model, in which the
$X_i$ are
uniformly distributed on $1,\ldots, N$, with $N=N_n\gg n$, such that, with
probability converging to $1$, there are no ties. In this case, the unknown
parameter $\bt$ is just the $N$-vector, $\bb= (\bt_1, \ldots, \bt_N)$.
Our goal is to estimate $\th= N^{-1} \sum_{i=1}^N \bt_i$.

To construct the prior, we proceed as follows. Assume that the components
$\bt_i$ are independent, with $\bt_i$ distributed according to a Beta
distribution with parameters $p_{\tau}(i)$ and $1-p_{\tau}(i)$, and
\[
p_{\tau}(i)=\frac{e^{\tau/\ga_i}}{1+e^{\tau/\ga_i}},
\]
with $\tau$ an unknown hyperparameter.
Let ${\th}^*=N^{-1}\cdot \sum_{i=1}^N p_{\tau}(i)$. Note that under the prior
$\th=N^{-1}\cdot \sum_{i=1}^N \bt_i = {\th}^*+O_{P}(N^{-1/2})$, by the CLT. We
now aim to estimate ${\th}^*$. In the language of \cite{lindley-smith-1972},
we shift interest from a random effect to a fixed effect. This is level 2
analysis in the language of \cite{everson-morris-2000}. The difference
between $\th$ and ${\th}^*$ is apparent in a full population analysis; see,
for example, \cite{berry-etal-1999} and \cite{li-1999}, where the real
interest is in~${\th}^*$.

In this simplified model, marginally, $X_1, \ldots, X_n$ are \iid
uniform on $1, \ldots, N$, $Y_i$ and $R_i$ are independent given $X_i$,
with $Y_i| X_i\sim\operatorname{Binomial} (1,p_{\tau}(X_i) )$, and
$R_i| X_i\sim\operatorname{Binomial} (1,g(X_i) )$.
The log-likelihood function for $\tau$ is given by
\[
\ell(\tau)=\sum_{R_i=1} \bigl[Y_i\log
p_{\tau}(X_i) +(1-Y_i)\log
\bigl(1-p_{\tau}(X_i) \bigr) \bigr].
\]
This is maximized at $\hat\tau$ satisfying
\begin{eqnarray*}
0 &=& n^{-1}\sum_{R_i=1}
\biggl(Y_i\frac{\dot p_{\hat\tau}(X_i)}{
p_{\hat\tau}(X_i)} - (1-Y_i) \frac{\dot p_{\hat\tau}(X_i)}{
1-p_{\hat\tau}(X_i)}
\biggr)
\\
&=& n^{-1} \sum_{R_i=1}\frac{\dot p_{\hat\tau}(X_i)}{
p_{\hat\tau}(X_i) (1-p_{\hat\tau}(X_i) )}
\bigl(Y_i-p_{\hat\tau}(X_i) \bigr)
\\
&=& n^{-1} \sum_{R_i=1}
\bigl(Y_i-p_{\hat\tau}(X_i) \bigr)/g(X_i)
\\
&=& \hat{\th}_{HT} - \frac{1}{n}\sum
_{i=1}^n \frac{R_i}{g(X_i)}p_{\hat\tau}(X_i),
\end{eqnarray*}
where $\dot p_\tau$ is the derivative of $p_\tau$ with respect to $\tau$.
A~standard {B}ernstein--von {M}ises argument shows that $\hat\tau$ is within
$o_{P}(n^{-1/2})$ of the Bayesian estimator of $\tau$, thus $\hat{\th}^*_B$,
the Bayesian estimator of ${\th}^*$, satisfies
\begin{eqnarray*}
\hat{\th}^*_B &=& \frac{1}{N}\sum
_{i=1}^N p_{\hat\tau}(i) + o_{P}
\bigl(n^{-1/2} \bigr)
\\
&=& \frac{1}{n}\sum_{i=1}^n
\frac{R_i}{g(X_i)} p_{\hat\tau}(X_i) + O_{P}
\bigl(n^{-1/2} \bigr)
\\
&=& \hat{\th}_{HT}+ O_{P} \bigl(n^{-1/2} \bigr)
\end{eqnarray*}
(where $O_P$ and $o_P$ are evaluated under the population model).

The estimator presented in \cite{li-2010} is somewhat similar; however, his
estimator is inconsistent, in general, and consistent only if $\E
(Y|R=1)=\E Y$
(as, in fact, his simulations demonstrate).

With this structure, it is unclear how to define sets of $\bt$ on which
uniform convergence holds. This construction merely yields an estimator
equivalent to the nonparametric HT estimator.

This prior produces a good estimator of ${\th}^*$ but, for other functionals,
for example, $\E (Y|g(X) > a )$ or $\E (\bb'\bb )$,
the prior leads
to estimators
which are not even consistent. So, if we are stuck with the resulting posterior,
as a type II Bayesian would be, we have solved the specific problem
with which
we were faced at the cost of failing to solve other problems which may
come to
interest us.

\section{The Partial Linear Model}\label{partial}

In this section, we consider the partial linear model, also known as the
partial spline model, which was originally discussed in \cite{engle-1986};
see also \cite{schick-1986}.
In this case, we have observations $W_i = (X_i,U_i,Y_i)$ such that
%
\begin{equation}
\label{partiallinear} Y_i = \th X_i + \bt(U_i) +
\ve_i,
\end{equation}
where the $(X_i,U_i)$ form an \iid sample from a joint density
$p(x,u)$, relative to Lebesgue measure on the unit square,
$[0,1]^2$; $\bt$ is an element of some class of functions $\mathcal{B}$;
and the $\ve_i$ are \iid standard-normal. The parameter of
interest is $\th$ and $\bt$ is a (possibly very nonsmooth) nuisance parameter.
Let $\ga(U)=\E (X | U )$. For simplicity, assume that $U$ is known
to be uniformly distributed on the unit interval.

\subsection{The Frequentist Analysis}

Up to a constant, the log-likelihood function equals
\[
\ell(\th,h,p) = - \bigl(y-\th x-\bt(u) \bigr)^2/2 - \log p(x,u).
\]
It is straightforward to argue that the score function for $\th$,
the derivative of the log-likelihood in the least favorable direction
for estimating $\th$ (cf. \citeauthor{schick-1986}, \citeyear{schick-1986}; \citeauthor{bickel-etal-1998}, \citeyear{bickel-etal-1998}),
is given by
\[
\tilde\ell_{\th} (\th,h ) = \bigl(x-\ga(u) \bigr) \bigl(y- \th x -
\bt(u) \bigr) = \bigl(x-\ga(u) \bigr)\ve,
\]
and that the semiparametric information bound for $\th$ is
\[
I = \E \bigl[\operatorname{var}(X|U) \bigr].
\]
We assume that $I>0$ (which implies, in particular, that $X$ is not a
function of $U$). Regarding estimation of $\th$, intuition based on
(\ref{partiallinear}) says that for small neighborhoods of $u$, the
conditional expectation of $Y$ given $X$ is linear with intercept
$\bt(u)$, and slope $\th$ which does not depend on the neighborhood.
The efficient estimator should average the estimated slopes over all
such neighborhoods.

Indeed, under some regularity conditions, an efficient estimator can be
constructed along the following lines. Find initial estimators $\tilde
\ga$
and $\tilde\bt$ of $\ga$ and $\bt$, respectively, and estimate $\th$ by
computing
\[
\hat{\th}= \frac{\sum (X_i-\tilde\ga(U_i)
 )
 (Y_i-\tilde\bt(U_i) )}{
\sum (X_i-\tilde\ga(U_i)
 )^2}.
\]
The idea here is that $\th$ is the regression coefficient associated with
regressing $Y$ on $X$, conditioning on the observed values of $U$. In order
for this estimator to be $\sqrt{n}$-consistent (or minimax), we need to
assume that the functions $\ga$ and $\bt$ are smooth enough that we can
estimate them at reasonable rates.

We could, for example, assume that the functions $\bt$ and $\ga$ satisfy
H\"older conditions of order $\alpha$ and   $\delta$, respectively.
That is, there is a constant $0\leq C<\infty$ such that
$|\bt(u)-\bt(v)|\leq C|u-v|^\alpha$
and $|g(u)-g(v)|\le C|u-v|^\delta$ for all $u,v$ in the support of $U$.
We also need to assume that $\operatorname{var}(X|U)$ has a version which is
continuous in $u$.
In this case, it is proved in \cite{wang-etal-2011} that a necessary
and sufficient condition for the existence of a $\sqrt{n}$-consistent and
semiparametrically efficient estimator of $\th$ is that $\alpha+\delta>1/2$.

\subsection{The Type I Bayesian Analysis}

We assume that the type I Bayesian places independent priors on
$p(u,x)$, $\bt$ and $\th$, $\pi=\pi_p\times\pi_{\bt}\times\pi_{\th}$.
For example, the prior on the joint density may be a function of
the environment, the prior on the nonparametric regression function
might be a function of an underlying physical process, and
the third component of the prior might reflect our understanding
of the measurement engineering. We have already argued that such
assumptions are plausible. The log-posterior-density is then
given by
\begin{eqnarray*}
\hspace*{-4pt}&&-\sum_{i=1}^n \bigl(Y_i-
\th X_i - \bt(U_i) \bigr)^2/2 + \log
\pi_{\th}(\th)+ \log\pi_{\bt}(\bt)
\\
\hspace*{-4pt}&&\quad{}+\sum_{i=1}^n \log
p(U_i,X_i) + \log\pi_{p}(p) + A,
\end{eqnarray*}
where $A$ depends on the data only. Note that the posterior
for $(\th,\bt)$ does not depend on $p$. The type I Bayesian would use
the same estimator regardless of what is known about the
smoothness of $\ga$.

Suppose now that, essentially, it is only known that $\bt$ is H\"older
of order $\alpha$, while the range of $U$ is divided up into intervals
such that, on each of them, $\ga$ is either H\"older of order $\delta_0$
or of order $\delta_1$, with
\[
\alpha+\delta_0 < 1/2 < \alpha+\delta_1.
\]
A $\sqrt{n}$-consistent estimator of $\th$ can only make use of data from
the intervals on which $\ga$ is H\"older of order of $\delta_1$. The rest
should be discarded. \emph{Suppose these intervals are disclosed to the
statistician.} If the number of observations in the ``good'' intervals
is of the same order as $n$, then the estimator is still
$\sqrt{n}$-consistent. For a frequentist, there is no difficulty in
ignoring the nuisance intervals---$\th$ is assumed to be the same everywhere.
However, the type I Bayesian cannot ignore these intervals. In fact, his
\emph{posterior} distribution cannot contain any information on
which intervals are good and which are bad.

More formally, let us consider a discrete version of the partial linear
model. Let the observations be $Z_i=(X_{i1},X_{i2},Y_{i1},Y_{i2})$, with
$Z_1,\ldots,Z_n$ independent. Suppose
\begin{eqnarray*}
X_{i1} &\sim& N(\ga_i, 1),
\\
X_{i2} &\sim& N(\ga_i+\eta_i,1),
\\
Y_{i1} &=&\th X_{i1}+\bt_i+\ve_{i1},
\\
Y_{i2} &=&\th X_{i2}+\bt_i+\mu_i+
\ve_{i2},
\\
\ve_{i1},\ve_{i2} &\stackrel{\mathrm{i.i.d.}} {\sim}&
N(0,1),
\end{eqnarray*}
where $X_{i1}, X_{i2},\ve_{i1}, \ve_{i2}$ are all independent, while
$\ga_i,\eta_i, \bt_i$, and $\mu_i$ are unknown parameters. We assume
that under the prior $(\ga_1,\eta_1),\ldots, (\ga_n,\eta_n)$ are \iid
independent of $\th$ and the $(\bt_1,\mu_1),\ldots,(\bt_n,\mu_n)$ are
\iid
This model is connected to the continuous version, by considering
isolated pairs of observations in the model with values differing
by $O(1/n)$. The H\"older conditions become $\eta_i=O_{P}(n^{-\delta_i})$,
and $\mu_i=O_{P}(n^{-\alpha})$, where $\delta_i\in\{\delta_0,\delta_1\}$,
as above.

From a frequentist point of view, the $(X_{i1},X_{i2},Y_{i1},\allowbreak  Y_{i2})$
have a joint normal distribution and we would then consider the statistic
\[
\lleft[ \matrix{ X_{i2}-X_{i1}
\cr
Y_{i2}-Y_{i1} } \rright] %
\sim N\lleft(
\lleft[ \matrix{ \eta_i
\cr
\th\eta_i+
\mu_i } \rright] %
, %
\lleft[ \matrix{ 2 & 2
\th
\cr
2\th& 2\th^2+2 } \rright] %
 \rright).
\]
Now consider the estimator
\begin{eqnarray*}
\hat{\th} &=& \frac{\sum_{\delta_i=\delta_1} (X_{i2}-X_{i1}) (Y_{i2}-Y_{i1})}{
\sum_{\delta_i=\delta_1}(X_{i2}-X_{i1})^2}
\\
&=& \th+ \frac{\sum_{\delta_i=\delta_1} (X_{i2}-X_{i1}) (\ve_{i2}-\ve_{i1})}{
\sum_{\delta_i=\delta_1} (X_{i2}-X_{i1})^2}
\\
&&{}+ \frac{\sum_{\delta_i=\delta_1} (X_{i2}-X_{i1}) \mu_i}{
\sum_{\delta_i=\delta_1} (X_{i2}-X_{i1})^2}
\\
&=& \th+ O_{P} \bigl(n^{-1/2} \bigr) + R,
\end{eqnarray*}
where
\[
R = \frac{\sum_{\delta_i=\delta_1} \eta_i \mu_i}{
\sum_{\delta_i=\delta_1} (X_{i2}-X_{i1})^2} = o_{P} \bigl(n^{-1/2} \bigr),
\]
since $\alpha+\delta_1 > 1/2$.

Note that if the sum were over all pairs, and if the number of pairs
with $\delta_i=\delta_0$ is of order $n$, then the estimator would not
be $\sqrt{n}$-consistent, since now $\sqrt{n} R$  diverges, almost surely.
In general, this model involves $2n+1$ parameters and the
parameter of \mbox{interest} cannot be estimated consistently unless the
nuisance parameters can be ignored, at least, asymptotically.
However, these parameters can only be ignored if we consider the smooth
pairs---that is, those pairs for which $\alpha+ \delta_i > 1/2$, making
the connection between variability, here, and smoothness, in the first
part of this section. Of course, the information on which pairs to use
in constructing the estimator is unavailable to the type I Bayesian.

The type I Bayesian does not find any logical contradiction in this
failure. The parameter combinations on which the Bayesian estimator fails
to be $\sqrt{n}$-consistent have
negligible probability, {a priori}. He assumes that {a priori},
$\bt$ and $\ga$ are independent and short intervals are essentially
independent since $\bt$ and $\ga$ are very rough. Under these assumptions,
the intervals on which $\ga$ is H\"older of order $\delta_0$ contribute,
on average, $0$ to the estimator. There are no data in these intervals
that contradict this {a priori} assessment. Hence, assumptions, made
for convenience in selecting the prior, dominate the inference. The trouble
is that, as discussed in Appendix~\ref{correlation}, even if we assume {a priori}
that $\bt$ and $\ga$ are independent, their cross-correlation may be
nonzero with high probability, in spite of the fact that this
random cross-correlation has mean $0$.

\section{The White Noise Model and Bayesian Plug-In Property}
\label{plugin}

We now consider the white noise model in which we observe the process
\[
dX(t) = \bt(t) \,dt + n^{-1/2}\,dW(t),\quad t\in(0,1),
\]
where $\bt$ is an unknown $L_2$-function and $W(t)$ is standard Brownian
motion. This model is asymptotically equivalent to models in density estimation
and nonparametric regression; see \cite{nussbaum-1996}
and \cite{brown-low-1996}.
It is also clear that this model is equivalent to the model in which we
observe
%
\begin{eqnarray}
\label{wnmodel}  X_i = \bt_i + n^{-1/2}
\ve_i,
\nonumber
\\[-8pt]
\\[-8pt]
\eqntext{\ve_i\stackrel{\mathrm{i.i.d.}} {\sim}
N(0,1), i = 1, 2, \ldots,} 
\end{eqnarray}
where $X_i$, $\bt_i$ and $\ve_i$ are the $i$th coefficients in an
orthonormal (e.g., Fourier) series expansion of $X(t)$, $\bt(t)$ and
$W(t)$, respectively. Note that  the entire sequence $X_1,X_2,\ldots$ is
observed, and $n$ serves only as a scaling parameter. We are interested
in estimating
$\bb=  (\bt_1, \bt_2, \ldots )$ as an object in $\ell_2$ with the
loss function $\|\hat{\bb}-\bb\|^2$ and linear functionals
$\th= \ga(\bb) = \sum_{i=1}^\infty\ga_i\bt_i$ with
$(\ga_1, \ga_2, \ldots)\in\ell_2$, also under squared error loss.
From a standard frequentist point of view, estimation in this problem is
straightforward. Simple estimators achieving the optimal rate of convergence
are given in the following proposition.

\begin{proposition}\label{linearminimax}
Assume that $\bb\in\mathcal{B}_{\alpha} =
 \{\bb: |\bt_i|\leq i^{-\alpha} \}$ and $\alpha> 1/2$.
The estimator $\widehat{\th}=\sum g_i X_i$ is $\sqrt{n}$-consistent for any
$g\in\ell_2$ and the estimator
\[
\widehat\bt_i = %
\cases{ X_i, &
$i^{\alpha}\leq n^{1/2}$,
\cr
0, & $i^{\alpha} >
n^{1/2}$, } %
\]
achieves the minimax rate of convergence, $n^{-(2\alpha-1)/2\alpha}$.
\end{proposition}

The proof is given in Appendix~\ref{proofs}.

\subsection{The failure of Type I Bayesian analysis}

A critical feature of Bayesian procedures for estimating linear functionals
is that they necessarily have the plug-in property (PIP). For example, for
squared error loss, since
\[
\E\ga(\widehat\bb) = \sum_{i=1}^n
\ga_i\E\widehat\bt_i,
\]
we have $\widehat{\ga(\bb)}=\ga(\widehat\bb)$, for any Bayesian estimators
of $\ga(\bb)$ and $\bb$ based on the same prior.

We say that $\widehat\bb$ is a \emph{uniformly efficient} plug-in
estimator for a set $\Theta$ of functionals and model $\mathcal{P}$ if
\[
\Bigl\{r_n^{-2}\|\widehat\bb- \bb\|_2^2
+ n\sup_{\th\in\Theta} \bigl(\th(\widehat\bb)-\th \bigr)^2
\Bigr\} = O_P(1),
\]
and $\widehat{\th}= \th(\widehat\bb)$ is semiparametrically efficient
for $\th$, where $r_n$ is the minimax rate for estimation of $\bb$.

\cite{bickel-ritov-2003} argued that there is no uniformly
efficient plug-in estimator in the white noise model when
$\Theta$ is large enough, for example, the set of all
bounded linear functionals. Every plug-in estimator fails
to achieve either the optimal \mbox{nonparametric} rate for estimating $\bb$
or $\sqrt{n}$-consistency as a plug-in-estimator (PIE) of at least one
bounded linear functional $\ga(\bb)$. The argument given in
\mbox{\citeauthor{bickel-ritov-2003}} (\citeyear{bickel-ritov-2003}) that no estimator with the PIP can
be uniformly efficient in the white noise model can be
refined slightly as follows.

We need the following lemma; the proof of which is given in
Appendix~\ref{proofs}.

\begin{lemma}\label{bayesbias}
Suppose $X\sim N(\bt,\sigma^2)$, $|\bt|\leq    a\leq\sigma$.
Let $\widehat\bt={\widehat\bt}(X)$ be the posterior mean when the prior
is $\pi$, assuming $\pi$ is supported on $[-a,a]$, and let $b_\bt$ be its
bias under $\bt$. Then $|b_\bt|+|b_{-\bt}| > 2(1-(a/\sigma)^2)|\bt|$.
In particular, if $\pi$ is symmetric about $0$, then
$|b_\bt| > (1-(a/\sigma)^2)|\bt|$.
\end{lemma}

This lemma shows that any Bayesian estimator is necessarily biased and puts
a lower bound on this bias. We use this lemma to argue that any
Bayesian estimator
will fail to yield $\sqrt{n}$-consistent estimators for at least one linear
functional.

\begin{theorem}\label{th4.3}
For any Bayesian estimator $\widehat\bb$ with\hspace*{-0.5pt} respect\hspace*{-0.5pt} to\hspace*{-0.5pt} prior\hspace*{-0.5pt} $\pi$\hspace*{-0.5pt} supported\hspace*{-0.5pt}
on\hspace*{-0.5pt} $\mathcal{B}_\alpha$\hspace*{-0.5pt}, with\hspace*{-0.5pt} $\alpha\hspace*{-0.5pt}>\hspace*{-0.5pt}1/2$, there is a pair
$(\ga,\bb)\in \ell_2\times\mathcal{B}_\alpha$ such that
$n [\ga(\widehat\bb)-\break \ga(\bb) ]^2\stackrel{p}{\rightarrow
}\infty$.
In fact, $\lim\inf_{n\rightarrow\infty}n^{(2\alpha-1)/4\alpha} [
\E_{\bb}\ga(\widehat\bb) -\break   \ga(\bb) ] > 0$.
\end{theorem}

\begin{pf}
It follows from Lemma~\ref{bayesbias} that for any $i>2n^{1/2\alpha}$
there are $\bt_i$ such that if $b_i=\E\widehat\bt_i-\bt_i$ then
$|b_i|>3i^{-\alpha}/4$. Define
\[
\ga_i = %
\cases{ 
0, & $i\leq2n^{1/2\alpha}$,
\cr
C n^{(2\alpha-1)/4\alpha} i^{-\alpha}, & $i > 2n^{1/2\alpha},
b_i > %
i^{-\alpha}/2$,
\cr
-C
n^{(2\alpha-1)/4\alpha} i^{-\alpha}, & $i > 2n^{1/2\alpha}, b_i <
-i^{-\alpha}/2$,} %
\]
where $C$ is such that $\sum_{i=1}^\infty\ga_i^2 = 1$. (Note that $C$
is bounded
away from $0$ and $\infty$.) We have
\begin{eqnarray*}
\E \Biggl[\sum_{i=1}^{\infty}\ga_i
(\widehat\bt_i - \bt_i ) \Biggr] &\geq& 3 C
n^{ (2\alpha-1)/4\alpha} \sum_{i > 2n^{1/2\alpha}} i^{-2\alpha}/4
\\
&\geq& 3 C n^{-(2\alpha-1)/4\alpha}/4.
\end{eqnarray*}
\upqed\end{pf}

Thus, any Bayesian estimator will fail to achieve optimal rates on some pairs
$(\bg,\bb)$. These pairs are not unusual. Actually they are pretty
``typical'' members of $\ell_2\times\mathcal{B}_\alpha$. In fact, for
any Bayesian
estimator $\widehat\bb$ and for almost all $\bb$ with respect to the
distribution with independent uniform coordinates on $\mathcal{B}_\alpha
$, there is a $\bg$ such that
$\ga(\widehat\bb)$ is inconsistent and asymptotically biased,
as in the theorem. Formally, let $\mu$ be a probability measure such that
the $\bt_i$ are independent and uniformly distributed on
$[-i^{-\alpha},i^{-\alpha}]$. Then, for any sequence of Bayesian estimators,
$\{{\widehat\bb}_n\}$,
\begin{eqnarray*}
&&\liminf_{n\rightarrow\infty} \mu \Bigl\{\bb: \sup_{\bg\in\ell_2}
n^{(2\alpha-1)/4\alpha} \bigl[\E_{\bb}\ga ({\widehat\bb}_n )
\\
&&\phantom{\liminf_{n\rightarrow\infty} \mu \Bigl\{\bb: \sup_{\bg\in\ell_2}
n^{(2\alpha-1)/4\alpha} \bigl[}{}-\ga(
\bb) \bigr] > M \Bigr\} = 1,
\end{eqnarray*}
for some $M > 0$. This statement follows from the proof of the Theorem~\ref{th4.3}, noting
that $\mu \{|b_i| > i^{-\alpha}/2 \} > 1/2$.

What makes the pairs that yield inconsistent estimators special, is
only that
the sequences $\bt_1, \bt_2, \ldots$ and $\ga_1, \ga_2, \ldots$ are nonergodic.
Each of them has a nontrivial autocorrelation function, and the two
autocorrelation functions are similar (see Appendix~\ref{correlation}).
The prior suggests that such pairs are unlikely and, therefore, that
the biases
of the estimators of each component cancel each other out. If the prior
distribution
represents a real physical phenomenon, this exact cancelation might be reasonable
to assume,
by the law of large numbers, and the statistician should not worry
about it.
If, on the other hand, the prior is a way to express ignorance or subjective
belief, then the analyst should worry about these small biases. This is
particularly true if the only reason for assuming that these small biases
are not going to accumulate is mathematical convenience. Indeed, in
high-dimensional spaces, auto-correlation functions may be complex,
with unknown neighborhood structures which are completely hidden
from the analyst.

We consider a Bayesian model to be \textit{honestly nonparametric} on
$\mathcal{B}_\alpha$, if the distribution of $\bt_i$, given $X_{-i}$,
is symmetric around $0$, and
$P(\bt_i>\epsilon i^{-\alpha} | X_{-i})>\epsilon$, for some $\epsilon>0$,
where $X_{-i}=X_1,\ldots,\break X_{i-1}, X_{i+1},\ldots$\,. That is, at least
in some
sense, all the components of $\beta_i$ are free parameters. In this case,
we have the following.

\begin{theorem}\label{the4.4}
Let the prior $\pi$ be honestly nonparametric on $\mathcal{B}_\alpha$ and
$1/2<\alpha<3/4$. Suppose $\bg=(\ga_1,\ga_2,\ldots) \in\mathcal
{B}_\alpha$, and
$\limsup\sqrt{n}\llvert \sum_{i={\nu n^{1/2\alpha}}}^\infty\ga_i\bt
_i\rrvert =\infty$
for some $\nu>1$. Then the Bayesian estimator of
$\ga(\bb)=\sum_{i=1}^\infty g_i\beta_i$
is not $\sqrt{n}$-consistent.
\end{theorem}

Note that if the last condition is not satisfied, then an estimator
that simply
ignores the tails ($i > n^{1/2\alpha}$) could be $\sqrt{n}$-consistent. However,
for $\bg,\bb\in\mathcal{B}_\alpha$, in general, all the first
$ n^{1/(4\alpha-2)}$ terms must be used, a number which is much greater
than $n^{1/2\alpha}$ for $\alpha$ in the range considered.

\begin{pf*}{Proof of Theorem~\ref{the4.4}}
Again, we consider the bias, as in the second part of Lemma~\ref{bayesbias}.
Under our assumptions, we have
\begin{eqnarray*}
&&\sqrt{n}\biggl\llvert \sum_{i > \nu n^{1/2\alpha}}\ga_i
(\E{\widehat\bt}_i -\bt_i )\biggr\rrvert
\\
&&\quad= \sqrt{n}\biggl\llvert \sum_{i > \nu n^{1/2\alpha}}(1-d_i)
\ga_i\bt_i\biggr\rrvert\quad \bigl(0\leq d_i\leq
n i^{-2\alpha}\bigr)
\\
&& \quad\geq\sqrt{n}\biggl\llvert \sum_{i > \nu n^{1/2\alpha}}
\ga_i\bt_i\biggr\rrvert - \sqrt{n}\sum
_{i > \nu n^{1/2\alpha}}n\llvert \ga_i\bt_i\rrvert
i^{-2\alpha
}
\\
&& \quad\geq\sqrt{n}\biggl\llvert \sum_{i > \nu n^{1/2\alpha}}
\ga_i\bt_i\biggr\rrvert - \sqrt{n}\sum
_{i > \nu n^{1/2\alpha}}n i^{-4\alpha}
\\
&& \quad= \sqrt{n}\biggl\llvert \sum_{i > \nu n^{1/2\alpha}}
\ga_i\bt_i\biggr\rrvert - o(1).
\end{eqnarray*}
\upqed\end{pf*}

Note that the assumptions of the theorem are natural if the prior
corresponds to
the situation in which the $\bt_i$ tend to $0$ slowly, so that we need
essentially
all the available observations to estimate $\ga(\bb)$ at the $\sqrt
{n}$-rate. As
in the last two examples, if either $\bt_i$ or $\ga_i$ converges to $0$ quickly
enough---that is, $\bt$ or $\ga$ are smooth enough---then the difficulty
disappears, as the tails do not contribute much to the functional $\ga
(\bb)$
and they can be ignored. However, when the prior is supported on
$\mathcal{B}_\alpha$, then the estimator ${\widehat\bt}_i = X_i$ is unavailable
to the Bayesian (whatever the prior) and $\ga(\bb)$ cannot be estimated
at the
minimax rate with $\bg\in\mathcal{B}_\alpha$, much less $\ell_2$.

\subsection{Type II Analysis}

It is easy to construct priors which give the global and local minimax rates
separately. For the nonparametric part $\bb$, one can select a prior for
which the $\bt_i$ are independent and the estimator of $\bt_i$ based on
$X_i\sim N(\bt_i,n^{-1})$ with $\bt_i$ restricted to the interval
$[-i^{-\alpha},i^{-\alpha}]$ is minimax; see \cite{bickel-1981}.
For the parametric part, one can use an improper prior under which the
$\bt_i$ are independent and uniformly distributed on the real
line. This prior works, but it completely ignores the constraints on the
coordinates of $\bb$. If one permits priors which are not supported on the
parameter space, then this prior is perfect, in the sense that any linear
functional can be estimated at the minimax rate.

If we are permitted to work with a prior which is not supported by the parameter
space, then we can construct a prior which yields good estimators for
both $\bb$
and any \emph{particular} linear functional. Indeed, suppose that $\ga
_i\neq0$,
infinitely often, and change bases so that $\tilde X = B'X$, where $B$
is an
orthonormal basis for $\ell_2$ with first column equal to $\bg/\|\bg\|
$. Note
that $\tilde X_1 = \sum_{j=1}^{\infty}\ga_j X_j / \|\bg\|$ and the
$\tilde X_i$
are independent, with $\tilde X_i\sim N (\tilde{\bt}_i,n^{-1} )$,
$i=0,1,\ldots$\,, where $\tilde{\bt}_1$ is the parameter of interest, and
$\|\tilde{\bb}\|_2=\|\bb\|_2$. Thus,\vspace{2pt} a Bayesian who places a flat
prior on $\th= \tilde{\bt}_1$ and a standard nonparametric prior on the
other coordinates of $\tilde{\bb}$, such that $\tilde{\bt}_i$ is estimated
by $\tilde X_i$, properly thresholded, will be able to estimate $\th$
efficiently and $(\tilde{\bt}_2, \tilde{\bt}_3, \ldots)$ at the minimax
rate, simultaneously; cf. \cite{zhao-2000}. Of course, this prior was
tailor-made for the specific functional $\th= \ga(\bb)$ and would yield
estimators of other linear functionals which are not $\sqrt{n}$-consistent,
should the posterior be put to such a task.

\begin{figure*}[b]

\includegraphics{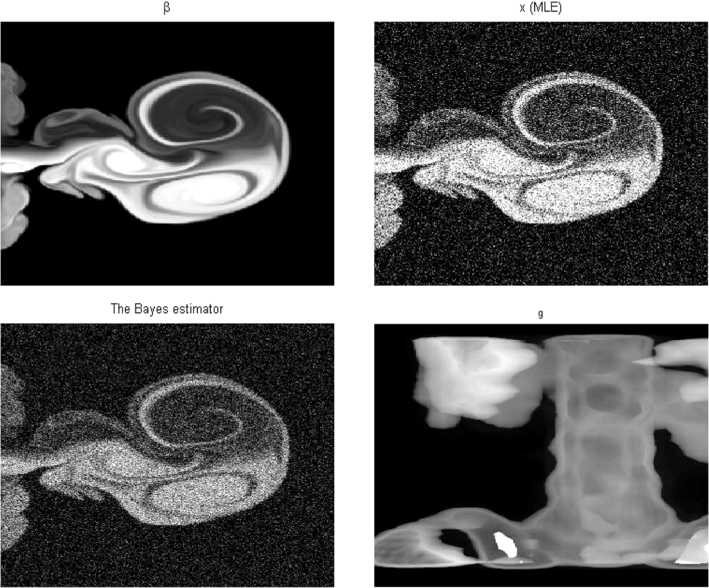}

\caption{Estimating linear functionals: \textup{(a)} the vector $\bb$; \textup{(b)} the
observations $X$; \textup{(c)} the Bayesian estimator; \textup{(d)} the functional~$g$.}\label{images}
\end{figure*}

\subsection{An Example}

To demonstrate that the effects described above have real, practical
consequences, consider the following example. Take $\bb= \operatorname{vec}(M_0)$
and $\bg= \operatorname{vec}(M_1)$, where $M_0$ and $M_1$ are the two images shown
in Figure~\ref{images}(a) and~(d), respectively. That is, each image
is represented by the matrix of the gray scale levels of the pixels,
and $\operatorname{vec}(M)$ is the vector obtained by piling the columns of $M$
together to obtain a single vector. These images were
sampled at random from the images which come bundled in the standard
distribution of \texttt{Matlab}. The images have been modified slightly,
so they both have the same $367\times300$ geometry, but nothing else
has been done to them. To each element of $\bb$, we added an independent
$N(0,169)$ random variable. This gives us $X$, shown in Figure~\ref{images}(b).
Let $\pi$ be that prior which takes the $\bt_i$ \iid $N(\mu,\tau^2)$,
where $\mu= \sum w_i\bt_i/\sum w_i$, with $w_i$ independent and identically
uniformly-distributed on $(0,1)$ and
$\tau^2 = 315.786$, the true empirical variance of the $\bt_i$.
The resulting nonparametric Bayesian estimator is shown in Figure~\ref{images}(c). The mean squared error (MSE) of this Bayesian estimator is
approximately
$65\%$ smaller than that of the MLE. Now consider the functional
defined by
$\bg$, shown in Figure~\ref{images}(d). Applying $\ga$ to $X$ yields an
estimator with root mean squared error (RMSE) of $1.04$, but plugging-in
the much cleaner Bayesian estimator of Figure~\ref{images}(c) gives an estimator
with a RMSE of $19.01$, almost twenty times worse than the frequentist
estimator.
Of course, the biggest difference between these two estimators is
bias: $0.01$ for the frequentist versus $19.00$ for the Bayesian.
These RMSE calculations were based on $1000$ Monte Carlo simulations.

There is no reason to suspect that these images are correlated---they
were sampled at random from an admittedly small collection of images---and they are certainly unrelated, one image shows the results of an
\mbox{astrophysical} fluid jet simulation and the other is an image of the
lumbar spine, but neither is permutation invariant nor ergodic, and
this implies that the two images may be strongly positively or
negatively correlated, just by chance; see Figure~\ref{corr} and
Appendix~\ref{correlation}.


\section{Estimating the Norm of a High-Dimensional Vector}\label{norm}

We continue with our analysis of the white noise model,
but we consider a different, nonlinear Euclidean parameter
of interest: $\th=\sum_{i=1}^\infty\bt_i^2$.

\begin{figure*}

\includegraphics{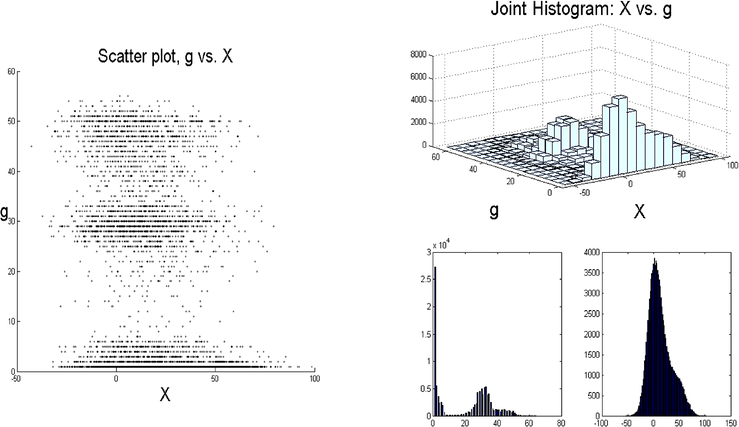}

\caption{A scatter plot and histograms of the data $X$ and functional
$g$. \textup{(a)} A scatter plot of $5\%$ of all pairs, chosen at random.
\textup{(b)} Joint and marginal histograms.}\label{corr}
\end{figure*}

A natural estimator of $\beta_i$ is given in Proposition~\ref{linearminimax}, and one may consider a plug-in estimator
of the parameter, given by
$\tilde{\th}=\sum\tilde\bt_i^2=\sum_{i<n^{1/2\alpha}}X_i^2$. This
estimator achieves the minimax rate for estimating $\bb$ and
$\tilde{\th}$ is an efficient estimator of the Euclidean parameter,
so long as $\alpha> 1$. But $\tilde\bt^2_i$ has bias $1/n$ as an
estimator of $\bt_i^2$. Summing $i$ from $1$ to $n$, we see that
the total bias is $n^{-1+1/2\alpha}$, which is much larger than
$n^{-1/2}$ if $\alpha< 1$. The traditional solution to this problem
is to simply unbias the estimator; cf. \cite{bickel-ritov-1988}.

\begin{proposition}\label{unbiased}
Suppose $3/4 < \alpha< 1$, then an efficient estimator of $\th$
is given by
%
\begin{equation}
\label{ub} \widehat{\th}= \sum_{i\leq m}
\bigl(X_i^2 - n^{-1} \bigr),
\end{equation}
for $n^{1/(4\alpha-2)} < m\leq n$.
\end{proposition}

\begin{pf}
Clearly, the bias of the estimator is bound\-ed by
\[
\sum_{i>m}i^{-2\alpha} < m^{-(2\alpha-1)} =
o_{P}\bigl(n^{-1/2}\bigr),
\]
and its variance is bounded by
\[
n^{-1}\sum_{i\leq m}\bigl(4
\bt_i^2+2/n\bigr)=4\th n^{-1}+o_{P}
\bigl(n^{-1}\bigr),
\]
demonstrating $\sqrt{n}$-consistency. The estimator is efficient since
$\hat{\th}$ is asymptotically normal, and $1/4\th$ is the semiparametric
information for the estimation of~$\th$.
\end{pf}

This is a standard frequentist approach: there is a problem and the
solution is justified because it works---it produces an asymptotically
efficient estimator of the parameter of interest---not because it fits
a particular paradigm. The difficulty with the naive, plug-in estimator
$\sum_{i\leq m}\hat\bt_i^2=\sum_{i\leq m} X_i^2$ is that it is biased,
but this is a problem that is easy to correct. Of course, this simple
fix is not available to the Bayesian, as we show next.

\subsection{The Bayesian Analysis: An Even Simpler Model}\label{simple}

We start with a highly simplified version of the white noise model.
To avoid confusion, we change notation slightly and consider
%
\begin{eqnarray}
\label{remodel} &Y_1, \ldots, Y_k \mbox{ independent with
} Y_i\sim N\bigl(\mu_i,\sigma^2\bigr),&
\\
&\th= \th(\mu_1,\ldots,\mu_k;g_1,
\ldots,g_k) = \displaystyle\sum_{i=1}^k
g_i\mu_i^2,&
\end{eqnarray}
where the $g_i$ are known constants. Here, we consider the asymptotic
performance of estimators of $\th$ with $\sigma^2 = \sigma
_k^2\rightarrow0$
as $k\rightarrow\infty$. Let
\[
\widehat{\th}= \sum_{i=1}^k
g_i \bigl(Y_i^2 -\sigma^2
\bigr).
\]
Clearly,
\[
\E\widehat{\th}= \th,\quad \operatorname{var}\widehat{\th} = 4
\sigma^2\sum_{i=1}^k
g_i^2\mu_i^2 + 2
\sigma^4\sum_{i=1}^k
g_i^2.
\]

Suppose that the $\mu_i$ are  {a priori} \iid $N(0,\tau^2)$,
with $\tau^2=\tau_k^2$ known, and consider the situation in which
$g_1\sim\cdots\sim g_k$. If $ k^{-1/2} \sigma_k^2\ll\tau_k^2\ll
\sigma_k^2$, then the signal-to-noise ratio $\tau^2/\sigma^2$
is strictly less than 1 and no estimator of $\mu_i$ performs much
better than simply setting $\widehat\mu_i = 0$. On the other hand,
$\widehat{\th}$ remains a good estimator of $\th$, with coefficient
of variation, $O  (\sqrt{k}\sigma^2/k\tau^2 )$, converging
to $0$. We call this paradoxical regime the \emph{nonlocalizable}
range, as we can estimate global parameters, like $\th$, but not
the local parameters, $\mu_1,\ldots,\mu_k$.

 {A posteriori}, the $\mu_i\sim N (\tau^2 Y_i/(\sigma^2 +\tau^2),
\tau^2\sigma^2/\allowbreak  (\sigma^2 +\tau^2) )$ and the Bayesian estimator of
$\th$ is given by
\begin{eqnarray*}
\sum_{i=1}^k g_i\E \bigl(
\mu_i^2 | Y_i \bigr) &=& \frac{\sigma^4+2\tau^2\sigma^2}{(\sigma^2+\tau^2)^2}
\sum_{i=1}^k g_i^2
\tau^2
\\
&&{}+ \frac{\tau^4}{(\sigma^2+\tau^2)^2}\sum_{i=1}^k
g_i \bigl(Y_i-\sigma ^2 \bigr).
\end{eqnarray*}
This expression has the structure of a Bayesian estimator in exponential
families: a weighted average of the prior mean and the unbiased estimator.
If the signal-to-noise ratio is small, $\tau^2\ll\sigma^2$, almost all
the weight is put on
the prior. This is correct, since the variance of $\th$, under the prior,
is much smaller than the variance of the unbiased estimator.
So, if we really believe the prior, the data can be ignored at little cost.
However, in frequentist terms, the estimator is severely biased and,
for a type II Bayesian, nonrobust.

The Achilles heel of the Bayesian approach is the~plug-in property.
That is, $\E (\sum_{i=1}^m\mu_i^2|\mbox{data} ) =    \sum_{i=1}^m
\E (\mu_i^2|\mbox{data} )$. However, when the signal-to-noise
ratio is infinitesimally
small, any Bayesian estimator must employ shrinkage. Note that, in particular,
the unbiased estimator $Y_i^2 - \sigma^2$ of $\mu_i^2$ cannot be Bayesian,
because it is likely to be negative and is an order of magnitude larger
than $\mu_i^2$.

A ``natural'' fix to the nonrobustness of the \iid prior, is to introduce
a hyperparameter. Let $\tau^2$ be an unknown parameter, with some
smooth prior. Marginally, under the prior, $Y_1,\ldots,Y_k$ are \iid
$N(0,\sigma^2+\tau^2)$. By standard calculations, it is easy to see that
the MLE of $\tau^2$ is
$\widehat\tau^2 = k^{-1}\sum_{i=1}^k (Y_i^2 - \sigma^2 )$.
By the Bernstein--von Mises theorem, the Bayesian estimator of $\tau^2$ must
be within $o_P(k^{-1/2})$ of $\widehat\tau^2$.
If $g_1 = \cdots= g_k$ and we plug $\tau= \widehat\tau$ into the
formula for
the Bayesian estimator, we get a weighted average of two estimators of
$\th$, both
of which are equal to $\widehat{\th}$. But, in general, $\widehat\tau$ is strictly
different from $\widehat{\th}$ and this estimator is inconsistent. Of
course, the
Bayesian estimator is not obtained by plugging-in the estimated value
of $\tau$,
but the difference would be small here, and the Bayesian estimator
would perform
poorly.

We can, of course, select the prior
so that the marginal variance is directly relevant to estimating $\th$. One
way to do this is to assume that $\tau^2$ has some smooth prior and, given
$\tau^2$, the $\mu_i$ are \iid $N(0,\allowbreak (\tau^2/g_i)-\sigma^2)$. Then
$Y_i\sim N(0,\tau^2/g_i)$, marginally, and the marginal log-likelihood
function is
\[
-k\log\bigl(\tau^2\bigr)/2 - \sum_{i=1}^k
g_i Y_i^2/2\tau^2.
\]
In this case, $\widehat\tau^2 = k^{-1}\sum_{i=1}^k g_i Y_i^2$ and
the posterior mean of $\sum_{i=1}^k g_i\mu_i^2$ is approximately
$\sum_{i=1}^k g_i (\widehat\tau^2/g_i -\sigma^2 ) =
\sum_{i=1}^k g_i (Y_i^2 -\sigma^2 )$, as desired.

This form of the prior variance for the $\mu_i$ is not accidental.
Suppose, more generally, that $\mu_i\sim N (0,\allowbreak  \tau_i^2(\rho) )$,
 {a priori}, for some hyperparameter $\rho$. Then the score equation
for $\widehat\rho$ is $\sum_{i=1}^k w_i(\widehat\rho) Y_i^2 =
\sum_{i=1}^k w_i(\widehat\rho)\cdot (\tau_i(\widehat\rho)+\sigma^2 )$,
where $w_i(\rho) = \tau_i(\rho)\dot\tau_i(\rho)/\break (\tau_i(\widehat
\rho)+
\sigma^2 )^2$. If we want the weight $w_i$ to be proportional to $g_i$,
then we get a simple differential equation, the general solution of which
is given by $ (\tau_i(\rho)+\sigma^2 )^{-1} = g_i\rho+ d_i$.
Hence, the general form of the prior variance is
\[
\tau_i^2(\rho) = (g_i\rho+
d_i )^{-1} -\sigma^2.
\]
The prior suggested above simply takes $d_i = 0$, for all~$i$. If the type
II Bayesian really believes that all the $\mu_i$ should have some known
prior variance $\tau_0^2$, he can take $d_i =  (\tau_0^2 +
\sigma^2 )^{-1} - g_i$, obtaining the expression
\[
\tau_i^2(\rho) = \frac{\tau_0^2 + (\rho-1)(\tau^2+\sigma^2)\sigma^2 g_i}{
1 + (\rho-1)(\tau^2+\sigma^2)\sigma^2 g_i}.
\]
If the variance of the $\mu_i$ really is $\tau_0^2$, then the posterior
for the hyperparameter $\rho$ will concentrate on $1$ and the $\tau_i^2$
will concentrate on $\tau_0^2$. If, on the other hand, $\tau^2$ is unknown,
the resulting estimator will still perform well, although the expression
for $\tau_i^2$ is quite arbitrary.

The discussion above holds when we are interested in estimating the
hyperparameter $\sum_{i=1}^k g_i\tau_i^2(\rho)$. This is a legitimate
change in the rules and the resulting estimator can be used to
estimate $\th$ in the nonlocalizable regime, because the main
contribution to the estimator is the contribution of the prior,
conditioning on $\tau_i^2(\rho)$. However, when
$\tau_i^2(\rho)\approx\sigma^2$, there may be a clear difference
between the Bayesian estimators of $\sum_{i=1}^k\tau_i^2(\rho)$ and
$\sum_{i=1}^k\mu_i^2$, respectively.

We conjecture that a construction based on stratification might be used
to avoid the problems discussed above: the use of an unnatural
prior and the difference between estimating the hyperparameter
and estimating the norm. In this case, we would stratify based
on the values of the $g_i$ and estimate $\sum\mu_i^2$ separately
in each stratum. The price paid by such an estimator is a large
number of hyperparameters and a prior suited to a very specific
task.

The discussion above shows that $\widehat{\th}$ can at least be
approximated by a Bayesian estimator, but the corresponding prior
has to have a specific form and would have to have been chosen
for convenience rather than prior belief. This presents no
difficulty for the type II Bayesian, who is free to select
his prior to achieve a particular goal. However, problems
with the prior remain. The prior is tailor-made for a
specific problem: while $\bt_1,\ldots,\bt_k$ \iid $N(0,\tau^2)$
is a very good prior for estimating $\sum_{i=1}^k\mu_i^2$, when
the parameter of interest is not permutation invariant, the
estimator is likely to perform poorly in frequentist terms.
Also, the prior is appropriate for regular models but not
sparse ones. Consider again the nonlocalizable regime in
which $\sqrt{k}\sigma^2\ll\th\ll k\sigma^2$, but suppose that
most of the $\mu_i$ are very close to zero, with only a few
taking values larger than $\sigma^2$ in absolute value.
A Bayesian estimator based on the prior suggested above will
shrink all the $Y_i$ toward $0$, strongly biasing the
estimates of the $\mu_i$, whereas a standard (soft or hard)
thresholding estimator will have much better performance.
A completely different prior is needed to deal with sparsity.
See \cite{greenshtein-etal-2008} and \cite{vanderpas-etal-2013}
for an empirical Bayes solution to the sparsity problem.

\subsection{A Bayesian Analysis of the White Noise Model}

Returning to original model, $X_i\sim N(\bt_i,1/n)$,
$|\bt_i|<i^{-\alpha}$, with $\th=\sum_{i=1}^n\bt_i^2$,
we can use a prior for which the $\bt_i$ are \iid
$N(0,\tau^2)$, for $i=1,\ldots, m$, and $0$, otherwise,
where $m=n^{1/(4\alpha-2)+\nu}$, for some $\nu> 0$.
This gives us a Bayesian estimator of $\th$ which is
asymptotically equivalent to the unbiased estimator,
$\widehat{\th}= \sum_{i=1}^n  (X_i^2 - n^{-1} )$,
and asymptotically efficient. However, the corresponding
estimator for $\bb$ is not even consistent and, when we
try to estimate $\bt_i$, even for $i$ relatively small,
we see that the Bayesian estimator shrinks $X_i$ toward $0$
by a factor of $1-\rho$ where $\rho$ is asymptotically
larger than $\th m/n=\th n^{-(4\alpha-3)/(4\alpha-2)-\nu} \gg n^{-1/2}$.
So our estimate of $\bt_i$ fails to be $\sqrt{n}$-consistent.

A more reasonable approach, in this situation, is to partition
the set $X_1,\ldots, X_n$ into blocks, $\{X_{k_{j-1}},\ldots
,\allowbreak  X_{k_j}\}$,
$j=1,\ldots,J$, and use a mean-zero Gaussian prior with unknown variance
in each of the blocks. One possible assignment is $k_0=1$, $k_1=o(\sqrt{n})$,
and $k_j=2k_{j-1}$, $j>1$. Thus, $O(\log n)$ blocks are needed. The analysis
presented above shows that this prior would yield a good estimator of
$\th$
without, hopefully, sacrificing our ability to estimate the $\bt_i$ at
the $\sqrt{n}$-rate.
Of course, this prior is not supported on the parameter space
$\mathcal{B}_{\alpha}$: it forces uniform shrinkage of the observations
in each block (and bypasses the plug-in property by estimating
block-wise hyperparameters). But there is nothing ``natural'' about
these blocks and nothing in the problem statement suggests this
grouping.

As before, this ``objective'' prior was constructed with a specific
parameter in mind and is unlikely to be effective for other parameters;
it cannot represent prior beliefs. The prior will also fail when
sparsity makes the block structure inappropriate. The unbiased,
frequentist estimator has no such difficulty. The Bayesian is
obliged to conform to the plug-in principle and, because of this,
at some stage, must get stuck with the wrong prior for some parameter
which was not considered interesting initially.

\begin{figure*}[b]

\includegraphics{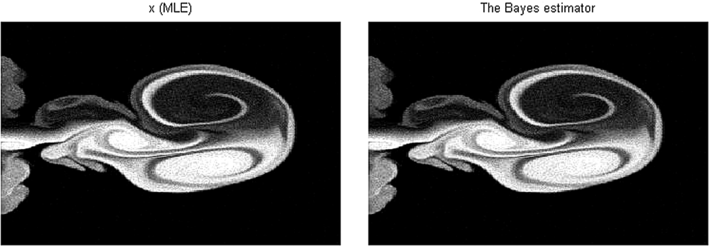}

\caption{\textup{(a)} the MLE; \textup{(b)} the Bayesian estimator.}\label{bayes-stop}
\end{figure*}

Consider a general prior $\pi$. Let $\pi_i$ be the prior for $\bt_i$
given $X_{-i}=(X_1, \ldots, X_{i-1}, X_{i+1}, \ldots)$. For $i>
n^{1/2\alpha+\nu}$
with $\nu> 0$ arbitrarily small and $m=n^{1/(4\alpha-2)+\nu}$, as in
Proposition~\ref{unbiased},
%
\begin{eqnarray}
\label{l2bi} &&\E_{\pi}\bigl(\beta_i^2 |
X_1,\ldots,X_m\bigr)
\nonumber
\\
&&\quad= \frac{\int_{-i^{-\alpha}}^{i^{-\alpha}} t^2
\varphi (n(X_i-t) )\,d\pi_i(t)}{
\int_{-i^{-\alpha}}^{i^{-\alpha}}\varphi (n(X_i-t) )\,d\pi_i(t)}
\\
&&\quad\in\bigl(a^{-1}\E_{\pi_i}
\beta^2_i,a\E_{\pi_i}\beta^2_i
\bigr),
\nonumber
\end{eqnarray}
where for $\mathcal{I}=\{i:n^{1/2+\nu}<i\le n^{1/(4\alpha-2)+\nu}\}$,
\[
\max_{i\in\mathcal{I}}\log a \leq\mathop{\max_{i\in\mathcal{I}}}_{|t_i|<i^{-\alpha}}
n\bigl\llvert (X_i-t_1)^2-(X_i-t_2)^2
\bigr\rrvert \stackrel{p} {\rightarrow} 0,
\]
since $\max_{i\in\mathcal{I}}n^{1/2-\nu}|X_i|\stackrel{p}{\rightarrow} 0$.
But this means that the estimate of $\beta_i^2$ depends only weakly on
$X_i$ itself. It is mainly a function of $X_{-i}$ and the prior. Moreover,
if the estimate of $\th$ is to be close to the unbiased one, then this
must be achieved through the influence of $X_i$ on the estimates of
$\bt_j$, for $j\neq i$. This is the case in the construction above
where formally we are estimating a hyperparameter of the prior,
rather than $\th$, itself. The result is a nonrobust estimator
which works for the particular functional of interest but not others.
In fact, we have the following theorem.

\begin{theorem}\label{l2fail}
Let $\bb_{-i}=(\bt_1,\ldots,\bt_{i-1},\bt_{i+1},\allowbreak  \ldots)$. Let $\pi$
be the
prior on $\bb$.
Suppose that there is an $\eta>0$ such that a.s. under the prior $\pi$:
$P_\pi(\lceil4i^{2\alpha}\beta_i^2\rceil= \kappa|\bb_{-i}) > \eta$,
$i=1,2,\ldots$\,, and $\kappa=1,\ldots,4$. There exists a set $S=S_n$ with
$\pi(S_n)\rightarrow1$, such that for all $\bb\in S$ there is a sequence
$g_1,g_2,\ldots$\,, for which the Bayesian estimator of $\sum g_i\bt
_i^2$ with
respect to $\pi$ is not $\sqrt{n}$-consistent.
\end{theorem}

The proof is given in  Appendix~\ref{proofs}. The conditions in the theorem
are needed to ensure that support of the  prior does not degenerate
to a
finite-dimensional parametric model.

\section{Data-Dependent Sample Sizes and Stopping Times}
\label{stopping}

The stopping rule principle (SRP) says that, in a sequential experiment,
with final data $\mathbf{x}^N(\tau)$, inferences should not depend on the
stopping time $\tau$; see \cite{berger-wolpert-1988}. In so much as Bayesian
techniques follow the strong likelihood principle (SLP), they must also
follow the SRP.

To see that high dimensional data represents a challenge for the SRP,
consider another version of the white noise model. Let
$n^{-2\alpha} < \bt_i < 3n^{-2\alpha}$,
$i=1,\ldots$\,, $k=\lfloor n^{2\alpha}\rfloor$, and
$1/6 < \alpha< 1/4$. Suppose that, for each $i$, $X_i(\cdot)$ is a
Brownian motion with drift $\bt_i$, and that $X_i$ is observed until
some random time $T_i$. Take $\bar X_i(t) = X_i(t)/t$ and note that
this is the sufficient statistic for $\bt_i$ given
$\{X_i(s) : s < t\}$. Of course, $\bar X_i$ is also
the MLE. Finally, let $\pi_i$ be the prior for $\bt_i$
given $X_{-i} = (X_1,\ldots,X_{i-1},X_{i+1},\ldots)$.
Let $f_i(\cdot)$ be the density of the distribution of
$X_i(T_i)$ given $X_{-i}$; $f_i = \pi_i * N(0,1/T_i)$.
We assume that the prior $\pi_i$ is non-parametric in the
sense that $\pi_i$ is bounded away from $0$ on the
allowed support, so that $X_{-i}$ does not give us
too much information about $\bt_i$.

It is well known that the posterior mean of $\bt_i$ given the
data satisfies
\[
\E(\bt_i | \mbox{data}) = \bar X_i(T_i) +
\frac{1}{T_i}\frac{f_i' (X_i(T_i) )}{f_i (X_i(T_i) )}.
\]

If $T_i = O(n)$, then $f_i\approx\pi_i$ and $\bar X_i(T_i)\approx\bt_i$.
Suppose $T_i$ is correlated with $g_i(\bt_i)$, where $g_i = f_i'/f_i$,
then the MLE of $\sum_{i=1}^k\bt_i$, given by $\sum_{i=1}^k X_i(T_i)$
is unbiased and has a random error on the order of $n^{\alpha} n^{-1/2}$,
while the Bayesian estimator has a bias which is $\sim n^{2\alpha}
n^{2\alpha}/n$,
with $n^{2\alpha}$ terms each contributing $O(n^{2\alpha})$ to the
bias, from
$g_i$, and a term of $O(1/n)$ from $1/T_i$.
With $1/6 < \alpha< 1/4$, the Bayes bias dominates the random error.

\subsection{An Example}

We consider again the vector $\bb$ represented in
Figure~\ref{images}(a), but this time the vectorized version
of the spine image shown in Figure~\ref{images}(d) is used to
specify the random number of observations associated with each
element of $\bb$. Adding noise to Figure~\ref{images}(a), we
get the observed data and MLE, shown in Figure~\ref{bayes-stop}(a).
This SNR is higher here than before ($+$2.72\,db) and, as a result,
the Bayesian estimator shown in Figure~\ref{bayes-stop}(b) is
much smoother.

Here, we used a prior with independent Gaussian components, each
with a mean equal to the mean of $\bb$ and variance equal to the
variance of the $\bt_i$. We have two processes on the unit square:
one represents $\bb$ and the other corresponds to random stopping
times, with the number of observations proportional to the gray-scale
value of the corresponding pixel in the image of the spine. As we have
already seen, these images are correlated, although there is no reason,
 {a priori}, to expect they would be, having been chosen at random
from a collection of unrelated images. This correlation causes trouble:
In $500$ Monte Carlo simulations, the RMSE of the Bayesian estimator of
the sum of the $\bt_i$ is $0.05$, whereas the RMSE of the MLE is
$0.009$. The difference is due almost entirely to bias. If we
replace the stopping times with a fixed time, the Bayesian estimator
performs better, achieving a RMSE of $0.0071$ versus the RMSE of
the MLE $= 0.0072$. This example shows clearly that the Bayesian
estimator can be badly biased when the stopping times and the
unknown parameters happen to be correlated.

\section{Bayesian Procedures are Efficient under Bayesian
Assumptions}\label{positive}

\cite{freedman-1965} proves that in some very weak sense consistency of
Bayesian procedures is ``rare.'' We, however, start with a version of Doob's
consistency result and show that the existence of a uniformly
$\sqrt n$-consistent estimator ensures that the posterior distribution
is $\sqrt n$-consistent with prior probability $1$.

To simplify notation, we consider the Markov chain
$\bt_0\rightarrow X_n\rightarrow\bt_n$, where $\bt_0, \bt_n\in\mathcal{B}$,
$\bt_0\sim\pi$, $X_n\sim P_{\bt_0}$, and given $X_n$, $\bt_0$ and $\bt_n$
are \iid That is, given $X_n$, $\bt_n$ is distributed according to the
posterior distribution $\pi_{X_n}$. Let $P$ be the joint distribution of
the chain. With some abuse of notation, $P_{\bt_0}$ is also the conditional
distribution of the chain given that it starts at $\bt_0$. Let $d_n$ be a
semi-metric on the parameter space, normalized to the sample size.
Typically, in the nonparametric situation considered in this paper,
$d_n(\bt,\bt')=\sqrt n|\th(\bt)-\th(\bt')|$ for some
real-valued functional $\th$ of the parameter.

We consider an estimator $\tilde\bt_n$ to be $d_n$ \emph{consistent
uniformly on} $\mathcal{B}$, if for every $\varepsilon> 0$ there is
an $M < \infty$ such that for all $\bt\in\mathcal{B}$ and $n$ large
enough, $P_\bt \{d_n(\tilde\bt_n,\bt)\geq M \}\leq\varepsilon$.
The posterior is $d_n$ \emph{consistent uniformly on} $\mathcal{B}$ if
for every $\varepsilon> 0$ there is an $M < \infty$ such that for
all $\bt_0\in\mathcal{B}$ and $n$ large enough,
$P_{\bt_0} \{d_n(\bt_n,\bt_0)\geq M \}\leq\varepsilon$.

Thus, we consider the inference to be $d_n$ uniformly consistent if
the frequentist Markov chain, $\beta_0\to X_n\to\tilde\beta_n$, or
the Bayesian one, $\beta_0\to X_n\to\beta_n$ lands in an
$O_p(1)$ $d_n$-ball about $\beta_0$.

\begin{theorem}\label{bayescons}
Suppose there is an estimator which is $d_n$ consistent uniformly on
$\mathcal{B}$. Then there is a $\mathcal{B}'\subseteq\mathcal{B}$ such
that $\pi(\mathcal{B}')=1$ and the posterior is $d_n$ consistent
uniformly on $\mathcal{B}'$.
\end{theorem}

The proof is given in Appendix~\ref{proofs}.

Thus, the existence of a uniformly good frequentist estimator ensures that
the there is a set with prior probability one such that the Bayes posterior
is uniformly consistent at the right rate on that set.
The difficulty with this statement is that, in high dimensional spaces,
there is no natural extension of Lebesgue measure and null sets of
very natural-looking priors are sometimes much larger than one would expect.
For a simple example, consider
a prior with hyperparameters of the type we considered for the white noise
models: $\tau$ has standard exponential distribution, and $\bt_i,\ldots
,\bt_k$
are, given $\tau$, \iid $N(0,\tau^2)$. Consider the set
$S=\{\bb: k^{-1}\sum_{i=1}^k (\bt_i-\bar\bt_k)^4 <
2.5 (k^{-1}\sum_{i=1}^k (\bt_i-\bar\bt_k)^2)^2\}$.
The probability of $S$ is $0.82$ if $k=5$. It drops to approximately $0.27$
when $k=50$. It is $0.0025$ for $k=500$, and negligible when $k=5000$.
(These numbers are based on $100\mbox{,} 000$ Matlab simulations.) The set $S$
is not so unusual or unexpected that it can be really ignored  {a priori}
and, unlike most sets, $S$ is simple to comprehend. If inferences
depend on
whether or not the fourth moment of the parameter is exactly three times
the square of the second, as implied by the normality assumption, which
was made for convenience, these inferences would not be robust.

Theorem~\ref{bayescons} does not contradict our findings. In the stratified
sampling and partial linear model examples of Sections~\ref{stratified} and~\ref{partial}, the difference between the Bayesian estimator and the frequentist
one is that the former ignores the information that restricts the model
to a
subset of the parameter space which has prior probability $0$. In the white
noise models of Sections~\ref{plugin} and~\ref{norm}, the requirement that
the prior be ``honestly nonparametric'' limits $\bt_1,\bt_2,\ldots$ to regular
sequences obeying a law of large numbers and, as a result, the set of
non-ergodic sequences is given prior probability $0$. And, in these examples,
there are two phenomena which make this theorem irrelevant. First, Bayesian
estimators must obey the plug-in principle, restricting estimators to those
of the form $\th(\tilde\bt)$ for $\tilde\bt\in\mathcal{B}$, while the
\mbox{frequentist} estimator cannot be written in this form. Second, each prior
fails for a different functional, but if the functional and the parameter
are chosen together, as we have argued might well happen, this theorem has
no consequences.

The second result of this section gives an easy abstract construction which
shows that, under some conditions, a type II Bayesian is able to choose
a prior with good frequentist properties. Our setup is as follows. In the
$n$th problem, we observe $X^{(n)}\sim P\in\mathcal{P}^{(n)}\ll\nu$, with
density $p=dP/d\nu$. Estimators take values in the set $\mathcal{A}$, and
a loss function
$\ell_n\dvtx \mathcal{P}^{(n)}\times\mathcal{A}\rightarrow\mathbb{R}^+$ is used
to assess the ``cost'' associated with a particular estimate. We assume that
$\ell_n$ is bounded by $L_n < \infty$ for all $n$ and that:
\begin{enumerate}[{A}3.]
\item[{A}1.]The loss function is Lipschitz: for all $a\in\mathcal{A}$ and
$P, P'\in\mathcal{P}^{(n)}$: $|\ell_n(P,a)-\ell_n(P',a)|\leq c_n\|
p-p_n\|$,
where $\|\cdot\|$ is the variational norm.
\item[{A}2.]Given $\varepsilon> 0$ there exists a finite set $\mathcal{P}_K^{(n)}
\subset\mathcal{P}^{(n)}$ with cardinality $\kappa_{n,\varepsilon}$,
such that
\[
\sup_{P\in\mathcal{P}^{(n)}}\inf_{P'\in\mathcal{P}_K^{(n)}} \bigl \|P-P'\bigr \|
\leq\varepsilon.
\]

\item[{A}3.]Let $R_n(P,\delta)=\E_P\ell_n(P,\delta(X))$, where
$\delta\dvtx \mathcal{X}^{(n)}\to\mathcal{A}$, or more generally, $\delta$ is
a randomized procedure (or Markov kernel from $\mathcal{X}^{(n)}$ to
$\mathcal{A}$). Let $R_n(\delta)=\sup_{P\in\mathcal{P}^{(n)}}
R_n(P,\delta)$.
There exist $\delta^*$ such that $R_n(\delta^*) = \inf_\delta
R_n(\delta)\equiv r_n\leq r<\infty$ for all $n$.
\end{enumerate}

Let $\mu_n$ be a probability measure on $\mathcal{P}_K^{(n)}$.
The corresponding posterior distribution is
$\mu_n(P_j|X^{(n)}) = \mu_n(P_j) p_j(X^{(n)})/\sum_{k=1}^{\kappa}
\mu_n(P_k) p_k(X^{(n)})$. Let $\delta_{\mu_n}$ be the Bayesian procedure
with respect to $\mu_n$.

\begin{theorem}\label{construction}
If conditions \textup{A1--A3} hold, then for all $\varepsilon' > 0$,
there exist $\mu_{n,\varepsilon'}$ on $\mathcal{P}^{(n)}$,
such that $R_n(\delta_{\mu_{n,\varepsilon'}})\leq r_n +
\varepsilon'$.
\end{theorem}

The proof is given in Appendix~\ref{proofs} and can be used to argue that,
under the conditions above, it is always possible (for a type II Bayesian)
to select a prior such that the corresponding Bayesian procedure estimates
both the global and local parameters at their minimax rates.

\begin{corollary}
Consider an estimation problem in which $\mathcal{P}^{(n)}$ satisfies the
conditions of Theorem~\ref{construction}; $\ell_{1n}(P,a)$, $\ell_{2n}(P,a)$
are two loss functions, each satisfying condition \textup{A1}, with Lipschitz constants
$c_{1n}$ and $c_{2n}$, respectively, and
\[
\inf_\delta\max_{P\in\mathcal{P}^{(n)}} \E_P
\ell_{kn}(P,\delta) = O\bigl(b_{kn}^{-1}\bigr),
\quad k = 1, 2,
\]
for some $b_{1n}, b_{2n}$.
Then, given $\varepsilon> 0$, there exist $\mu_n$ on $\mathcal{P}^{(n)}$
such that simultaneously
\[
\max_{P\in\mathcal{P}^{(n)}} \E_P\ell_{kn}(P,
\delta_{\mu_n}) = O\bigl(b_{kn}^{-1}\bigr),\quad k =
1, 2.
\]
\end{corollary}

The corollary follows by applying the theorem to the combined loss function
$\ell_n(P,(a_1,a_2)) = b_{1n}\ell_{1n}(P,\allowbreak  a_1) + b_{2n}\ell_{2n}(P,a_2)$.

The conditions essentially hold in our examples (technically, in the
stratified sampling and partial linear model examples, before applying
the theorem, one should restrict the parameter space to a compact set).
However, note that the prior may depend on information that may not be
known
{a priori}, such as the loss function, and on parameters that
``should not'' be part of the loss, such as the weight function in the
stratified sampling example, the (smoothness of the) conditional
expectation of $U$ given $X$ in the partial linear model, and the
linear functional in the white noise model.

Note, however, that the theorem as proven does not say that there exists
a prior such that the two Bayesian estimators for each of the two loss
functions achieve the corresponding minimax rates. Indeed, a single
estimator is produced which balances the two objectives.

\section{Summary}
\label{summary}

In this paper, we presented a few toy examples in which a nonparametric
prior fails to produce estimators of simple functionals that
are $\sqrt{n}$-consistent, in spite of the fact that efficient frequentist
procedures exist (and are often easy to construct). In these examples,
minimal smoothness was assumed, but we do not believe that this is
necessary in order for the Bayesian paradigm to have difficulty with
high-dimensional models. With minimal smoothness, it is easy to
prove that bias accumulates and global functionals cannot be
estimated at minimax rates (while with smoother objects, this would
be more difficult to demonstrate).

Bayesian procedures are always unbiased with the respect to the prior on
which they are based. Bayesian estimators tend to replace parameters buried
in noise by their  {a priori} means. This would be a reasonable strategy
if the prior represented a physical reality, but is not workable if the prior
represents a subjective belief or is selected for computational convenience.
In the latter case, to the extent that the beliefs or assumptions fail to
match the physical reality, the Bayesian paradigm will run into difficulty.

Several difficulties with the Bayesian approach were demonstrated by
our examples, including:
\begin{enumerate}[3.]
\item[1.] The possibility of  {de facto} cross-correlation
between two independent processes, as discussed in
Appendix~\ref{correlation},
is ignored by the Bayesian estimator. The effect of such \emph{spurious
correlations} can be seen in the stratified sampling example of
Section~\ref{stratified}, the partial linear model of
Section~\ref{partial}, and the discussion on estimating linear
functionals in the white noise model of Section~\ref{plugin}. Because
the spurious correlations observed have mean value $0$, the Bayesian
estimators are unbiased, on average, but this average is only with
respect to the prior. In any other sense, the Bayesian estimators are biased.
\item[2.] For linear functionals with squared error loss, the Bayesian
paradigm requires the analyst to follow the \emph{plug-in principle},
estimating functionals $\th$ of high-dimensional parameters $\bt$
by $\tilde{\th}= \th(\tilde\bt)$. The fact that universal plug-in
estimators do not exist shows that strict adherence to the Bayesian
paradigm is too rigid. This was shown in Section~\ref{plugin}.
\item[3.] Having selected a prior, the Bayesian may assume that some
functionals of the unknown parameter are known---for example, weighted
means of many unknown parameters. But, as a matter of fact, these
unverified assumptions, hidden in the selected prior, force the
resulting estimator to be \emph{nonrobust}. See, for example, the
discussion of the partial linear model in Section~\ref{partial}.
\item[4.] On the other hand, replacing components of signal buried
deeply in noise by their prior means may cause an \emph{accumulation
of bias}, destroying estimators of functionals which can be estimated
without bias and with bounded asymptotic variance. This is clear from
the discussion in Section~\ref{norm}.
\item[5.] Finally, the Bayesian paradigm forces the analyst to follow the
strict likelihood principle, cf. \cite{berger-wolpert-1988}, and this may
force him to \emph{ignore auxiliary information} which could be used to produce
asymptotically unbiased, efficient estimators. This was the core of the argument
in the stratified sampling example of Section~\ref{stratified} in which the
type~I Bayesian cannot make use of information on the sampling probabilities,
at all, and cannot produce a $\sqrt{n}$-consistent estimator of the population
mean, in general, as a result. The same is true in the partial linear model
example of Section~\ref{partial}, in which the type I Bayesian analyst cannot make
use of information on smoothness, and in the stopping times example of
Section~\ref{stopping}.
\end{enumerate}

Real-life examples are more complex and less tracta\-ble than the toy
problems we have played with in this paper and, as a result, it would be
more difficult to determine the real-life effect of assumptions hidden
in the prior on the frequentist behavior of Bayesian estimators in such
situations. It is very difficult to build a prior for a very complicated
model. Typically, one would assume a lot of independence. However, with
many independent or nearly-independent components, the law of large numbers
and central limit theorem will take effect, concentrating what was
supposed to have been a vague prior in a small corner of the parameter
space. The
resulting estimator will be efficient for parameters in
this small set, but not in general. It is safe to say that Bayes is
not curse of dimensionality appropriate (or CODA, see
\citeauthor{robins-ritov-1997}, \citeyear{robins-ritov-1997}).

\begin{appendix}

\section{Independent but Correlated Series}\label{correlation}

Much of the analysis in this paper is based on presenting counterexamples
on which a given estimation procedure fails. This is satisfactory from a
minimax frequentist point of view: one example is enough to argue that the
result depends on the unknown parameter and is not uniformly valid, or
asymptotically minimax. However, this may not convince a Bayesian, who
might claim that the counter example is  {a priori} unreasonable. A
typical example of the argument was presented in the stratified sampling
example of Section~\ref{stratified}. This argument can be characterized
by constructing two  {a priori} independent processes ($\bt$ and
$\ga$),
which happen to be ``similar.'' For the Bayesian, this is a very unlikely
event. After all, he assumes that they are independent; for example, one
of them depends on biology and the other on budget constraints. In this
section, we argue that such correlations can actually be quite likely.
\cite{harmeling-toussaint-2007} write: ``Let us now get to the core of
\cite{robins-ritov-1997}. The authors consider uniform unbiasedness of
an estimator. This means that the estimator has to be unbiased for every
possible choice of $\th$ and $\xi$. In the experiment, we performed above,
though, we chose $\th$ and $\xi$ independently, and thus it was very
unlikely that we ended up with an accidentally correlated $\th$ and
$\xi$, for example, where $\th$ tends to be large whenever also $\xi$
is (or
inversely).'' (We should remark that they consider also a scenario in
which the processes are correlated.) We claim that this criticism ignores
the fact that two processes can be independent and yet, with high
probability, have an empirical cross-correlation which is far from $0$.
This would be the case, for example, if the processes are nonergodic
and have similar autocorrelation functions.

Suppose $U_1,\ldots,U_n$ and $V_1,\ldots,V_n$ are two independent simple
random walks. Then of course $U_n$ and $V_n$ are uncorrelated. But we
may consider the correlation between these two series
$R=n^{-1}\sum_{i=1}^n (U_i-\bar U_n)(V_i-\bar V_n)$, where $\bar U_n$ and
$\bar V_n$ are the empirical means of the two series, respectively.
$R$ is a random variable and clearly it has mean 0. However, it is far
from being close to $0$, even if $n$ is large. In fact, asymptotically,
$R$ is almost uniformly distributed on most of the interval $(-1,1)$;
cf. \cite{mcshane-wyner-2011}. The reason for this somewhat surprising
fact is that random walks and Brownian motions are less wild than they
are sometimes thought to be. In fact given $U_n$, the best predictor of
$U_{\lfloor n/2\rfloor}$ is $U_n/2$, where $\lfloor a\rfloor$ is the
largest integer less than $a$, and the sequence tends to be, very
roughly speaking, monotone. But if both $U_1,\ldots,U_n$ and
$V_1,\ldots,V_n$ are ``somewhat'' monotone, then they will be
cross-correlated; maybe positively correlated, maybe negatively, but
rarely uncorrelated. Consider now two general, independent mean $0$ random,
nonmixing sequences $U_1,\ldots,U_n$ and $V_1,\ldots,V_n$. Suppose that
the two sequences have the autocorrelation functions
$A(i,j)=\operatorname{cov}(U_i,U_j)$ and $B(i,j)=\operatorname{cov}(V_i,V_j)$,
where we assume $\operatorname{var}(U_i) = \operatorname{var}(V_i) = 1$ (although, in
the standard engineering usage, autocorrelation refers to what some
would like to call autocovariance).
We do not assume that the series are stationary, and we do not know their
autocorrelation functions. The picture we have in mind is that each
$(U_i,V_i)$ is a characteristic of points in a large graph, and neighboring
nodes are highly correlated, but we do not know the neighborhood structure
of the graph. Define
\[
R = \langle U,V\rangle_0 \equiv n^{-1}\sum
_{i=1}^n U_i V_i -
n^{-2}\sum_{i=1}^n
U_i\sum_{i=1}^n
V_i,
\]
where $\langle\cdot,\cdot\rangle_0$ is the empirical cross-covariance between
two sequences. Then $\E R=0$, while direct calculations give
\begin{eqnarray*}
\operatorname{var}(R) &=& n^{-1} \sum_{i=1}^n
\bigl\langle A(i,\cdot), B(i,\cdot) \bigr\rangle_0
\\
&&{}- \Biggl\langle
n^{-1}\sum_{j=1}^n A(\cdot,j),
n^{-1}\sum_{j=1}^n B(\cdot,j)
\Biggr\rangle_0.
\end{eqnarray*}
To get some sense of the size of $\operatorname{var}(R)$, suppose that
$n^{-1}\sum_{j=1}^n A(i,j)\equiv n^{-1}\sum_{j=1}^n B(i,j)\equiv c$.
Then we get
\[
\operatorname{var}(R) = n^{-2}\sum_{i=1}^n
\sum_{j=1}^n \bigl(A(i,j)-c \bigr)
\bigl(B(i,j)-c \bigr).
\]
Clearly, if the two series are mixing and $\sum_j A(i,j)=\sum_jB(i,j)=O(1)$,
then $\operatorname{var}(R)=O(n^{-1})$. However, if they are not mixing, and have
similar autocorrelation functions, then most realizations of these two series
will have nonzero cross-correlation.

\section{Proofs}\label{proofs}

\setcounter{equation}{7}
\renewcommand{\theequation}{\arabic{equation}}

\begin{pf*}{Proof of Proposition~\ref{linearminimax}}
Clearly,
\begin{eqnarray*}
\E\sum_{i=1}^{\infty} (\widehat
\bt_i-\bt_i )^2 &=& \bigl\lfloor
n^{1/2\alpha}\bigr\rfloor/n + \sum_{i > n^{1/2\alpha}}
\bt_i^2
\\
&\leq& n^{-(2\alpha-1)/2\alpha} + \sum_{i > n^{1/2\alpha}}
i^{-2\alpha}
\\
&\leq&2\alpha n^{-(2\alpha-1)/2\alpha}/(2\alpha- 1).
\end{eqnarray*}
That this is the minimax rate is established by considering the
prior $\Pi$ which makes $\bt_1,\bt_2,\ldots$ independent, with
$\Pi (\bt_i = \pm i^{-\alpha} ) = 1/2$.
\end{pf*}

\begin{pf*}{Proof of Lemma~\ref{bayesbias}}
First note that because of the monotone likelihood ratio property,
$\widehat{\th}(x)$ is a monotone increasing function of $x$. We have
\begin{eqnarray*}
1 + \dot b_{\th} &=& \partial\E_{\th}\E_{\pi} (
\Theta| X ) /\partial\th
\\
&=& \frac{\partial}{\partial\th} \E_{\th}\frac{\int t e^{-(X-t)^2/2\sigma^2}\,d\pi(t)}{
\int e^{-(X-t)^2/2\sigma^2}\,d\pi(t)},
\end{eqnarray*}
where $\E_{\th}$ is the expectation assuming the true value of the parameter
is $\th$, $(\Theta,X)$ is a pair of random variables such that $\Theta
\sim\pi$,
and $X|\Theta\sim N(\Theta,\sigma^2)$ and $\E_\pi$ is the expectation under
this joint distribution. Note that $\E_\pi$ is a formal expression
since we
assume that $X\sim N(\th,\sigma^2)$. Let $Z\sim N(0,\sigma^2)$ then
\begin{eqnarray*}
&&1 + \dot b_{\th}
\\
&&\quad =  \frac{\partial}{\partial\th} \E_{\th}
\frac{\int t e^{-(Z+\th-t)^2/2\sigma^2}\,d\pi(t)}{
\int e^{-(Z+\th-t)^2/2\sigma^2}\,d\pi(t)}
\\
&&\quad =  \frac{1}{\sigma^2} \E \biggl\{\frac{\int t(t-Z-\th) e^{-(Z+\th-t)^2/2\sigma^2}\,d\pi(t)}{
\int e^{-(Z+\th-t)^2/2\sigma^2}\,d\pi(t)}
\\
&&\phantom{\quad =  \frac{1}{\sigma^2} \E \biggl\{}{} - \frac{\int t e^{-(Z+\th-t)^2/2\sigma^2}\,d\pi(t)}{
\int e^{-(Z+\th-t)^2/2\sigma^2}\,d\pi(t)}
\\
&&  \phantom{\quad =  \frac{1}{\sigma^2} \E \biggl\{-\,\,\,}{}\cdot\frac{\int(t-Z-\th) e^{-(Z+\th-t)^2/2\sigma^2}\,d\pi(t)}{
\int e^{-(Z+\th-t)^2/2\sigma^2}\,d\pi(t)} \biggr\}
\\
&&\quad =  \frac{1}{\sigma^2} \E_{\th}\bigl\{\operatorname{var}(\Theta| X)
\bigr\}.
\end{eqnarray*}
Hence, $0\leq1+\dot b_{\th}\leq(a/\sigma)^2$, or $\dot b_{\th}\in
[-1,-(1-(a/\sigma)^2)]$.
The lemma then follows from the mean value theorem.
\end{pf*}

\begin{pf*}{Proof of Theorem~\ref{l2fail}}
Let $\bb\sim\pi$. For any $i$, let $F_i$ be the distribution of
$b_i=\E ({\bt_i^2|X_{-i}} )$. Note that $b_i$ and $\bt_i$
are independent given $\bb_{-i}$. By assumption, conditionally on
$\bb_{-i}$, $P_{\pi_i\times F_i}(|\bt_i^2-b_i|>i^{2\alpha}/4)>\eta$.
But then it follows from \eqref{l2bi} that for $n$ large enough,
$P_\pi(|\bt_i^2-\widehat\bt_i^2|>i^{2\alpha}/4)>\eta/2$.
Let $c_i'(\bb) = \mathbf{1} \{E_{\bb}(\widehat\bt^2_i-\bt^2_i)
< -i^{2\alpha}/4 \}$, $c_i''(\bb) = \mathbf{1} \{E_{\bb}
(\widehat\bt^2_i-\bt^2_i) > i^{2\alpha}/4 \}$, and
\[
\hspace*{10pt}c_i(\bb) = %
\cases{ c_i'(\bb),
\cr
\quad\hspace*{12pt}\displaystyle\sum_{i=n^{1/2\alpha+ \nu}}^m c_i'(
\bb) > \eta /3\bigl(m-n^{1/2\alpha+ \nu}\bigr),
\cr
c_i''(
\bb), \quad \mbox{otherwise}. } %
\]
Now, $\sum c_i(\bb)\widehat\bt_i^2$ picks exactly those $\bt_i^2$ which are
estimated with bias, positive bias if $c_i = c_i''$ and negative if
$c_i = c_i'$.
\end{pf*}

\begin{pf*}{Proof of Theorem~\ref{bayescons}}
The proof is based on the two lemmas which follow.
Suppose the posterior is not $d_n$ consistent on $\mathcal{B}'$ with
$\pi(\mathcal{B}') > 0$. Then,  by Lemma~\ref{bayescons1}, \eqref{lemc1}
must hold for $\bt_0\in\mathcal{B}'$. By Lem\-ma~\ref{bayescons2}, \eqref{lemc22}
must hold. But \eqref{lemc22} contradicts $\pi(\mathcal{B})=1$, since
then, for
all $M$, we have $\pi\{\bt: P_\bt(d_n(\tilde\bt_n,\bt)\ge M)\} > 0$.
\end{pf*}

Recall that $\bt_0$ is the true parameter. It has a prior probability
$\pi$.
$\bt_n$ is a random variable which, given the data $X_n$, has the posterior
distribution $\pi_{X_n}$. The first lemma says that if there is a $d_n$
consistent
estimator, but $d_n(\bt_n,\bt_0)$ is not tight, then neither is
$d_n(\bt_n,\tilde\bt_n)$.

\begin{lemma}\label{bayescons1}
Suppose that:
\begin{enumerate}[2.]
\item[1.]
There\hspace*{-0.5pt} is\hspace*{-0.5pt} a\hspace*{-0.5pt} statistic\hspace*{-0.5pt} $\tilde\bt_n$\hspace*{-0.5pt} such\hspace*{-0.5pt} that\hspace*{-0.5pt}
$\limsup_n\hspace*{-0.5pt} P_{\bt_0} (d_n(\tilde\bt_n,\allowbreak  \bt_0)\geq M
)\rightarrow0$
as $M\rightarrow\infty$.
\item[2.] For all $M<\infty$,
$\limsup_n P_{\bt_0} (\pi_{X_n}(d_n(\bt_n,\bt_0)\geq2M)\geq
2\varepsilon )\geq2d$.
\end{enumerate}
Then there is an $M$ which may depend on $\bt_0$ such that
%
\begin{equation}
\label{lemc1} \limsup_{n\rightarrow\infty} P_{\bt_0} \bigl(
\pi_{X_n}\bigl(d_n(\bt_n,\tilde
\bt_n)\geq M\bigr) \geq\varepsilon \bigr)\geq d.
\end{equation}
\end{lemma}

\begin{pf}
\begin{eqnarray*}
&&P_{\bt_0} \bigl(\pi_{X_n}\bigl(d_n(
\bt_n,\tilde\bt_n)\geq M\bigr)\geq\varepsilon \bigr)
\\
&&\quad \geq P_{\bt_0} \bigl(\bigl\{\pi_{X_n}
\bigl(d_n(\bt_n,\bt_0)\geq2M\bigr)\geq 2
\varepsilon\bigr\}
\\
&&\phantom{\quad \geq P_{\bt_0} \bigl(}\cap\bigl\{d_n(\tilde\bt_n,
\bt_0)\leq M\bigr\} \bigr)
\\
&&\quad \geq P_{\bt_0} \bigl(\pi_{X_n}\bigl(d_n(
\bt_n,\bt_0)\geq2M\bigr)\geq2\varepsilon \bigr)
\\
&&\phantom{\quad \geq}{} -
P_{\bt_0} \bigl(d_n(\tilde\bt_n,
\bt_0)\geq M \bigr).
\end{eqnarray*}
By assumption, the lim\,sup of the first term on the right-hand side is bounded
by $2d$, while we can choose $M$ large enough such that the second term on
the right-hand side is bounded by $d$ for all $n$ large enough. The
lemma follows.
\end{pf}

The reverse is given in the following lemma.

\begin{lemma}\label{bayescons2}
Suppose there is a statistic $\tilde\bt_n$ and $M, \varepsilon, d > 0$
such that
%
\begin{equation}
\label{lemc21} P_{\bt_0} \bigl(\pi_{X_n}\bigl(d_n(
\tilde\bt_n,\bt_n)\geq M\bigr) \geq\varepsilon \bigr)\geq
d
\end{equation}
for all $\bt_0\in\mathcal{B}'$ and $\pi(\mathcal{B}')\geq\gamma> 0$.
Then for all $M < \infty$,
%
\begin{equation}
\label{lemc22} P\bigl(d_n(\tilde\bt_n,\bt_0)
\geq M\bigr)\geq\varepsilon \,d\gamma.
\end{equation}
\end{lemma}

\begin{pf}
If $U,V,W$ are three random variables, then $E(E(E(U|V)|W))=E(U)$.
Computing the expected value of \eqref{lemc21}, we obtain \eqref{lemc22}.
\end{pf}

\begin{pf*}{Proof of Theorem~\ref{construction}}
Let $P,P'\in\mathcal{P}^{(n)}$. Then
\begin{eqnarray*}
&&\bigl |R_n(P,\delta) -R_n\bigl(P',\delta\bigr)\bigr |
\\
&&\quad= \biggl\llvert \int\ell \bigl(p, \delta(x) \bigr) p(x) \,d\nu(x)
\\
&&\phantom{\quad= \biggl\llvert}{}- \int
\ell \bigl(p',\delta(x) \bigr) p'(x) \,d\nu(x)\biggr
\rrvert
\\
&&\quad \leq\int\bigl\llvert \ell\bigl(p,\delta(x)\bigr) - \ell
\bigl(p',\delta(x)\bigr)\bigr\rrvert p(x) \,d\nu(x)
\\
&&\qquad{} + \int\ell\bigl(p',\delta(x)\bigr)\bigl\llvert
p(x)-p'(x)\bigr\rrvert \,d\nu(x).
\end{eqnarray*}
The first term on the right-hand side is bounded by $c_n\|P-P'\|$, and
the second by $L_n\|P-P'\|$, so that
%
\begin{equation}
\label{pb2} \qquad\bigl |R_n(P,\delta)-R_n\bigl(P',
\delta\bigr)\bigr | \leq(c_n+L_n)\bigl \|P-P'\bigr \|,
\end{equation}
for all $\delta$, $P$, and $P'$.

By the complete class theorem, for any $\varepsilon' > 0$ there is
a $\mu_n$ supported on $\mathcal{P}_K^{(n)}$ such that
%
\begin{equation}
\label{pb3} \quad\max_{P\in\mathcal{P}_K^{(n)}} R_n(P,
\delta_{\mu_n}) \leq\inf_{\delta}\max_{P\in\mathcal{P}_K^{(n)}}
R_n(P,\delta) + \varepsilon'.
\end{equation}
By \eqref{pb2}, we also have
%
\begin{eqnarray}
\label{pb4} &&\Bigl |\max_{P\in\mathcal{P}_K^{(n)}} R_n(P,\delta) - \max
_{P\in\mathcal{P}^{(n)}} R_n(P,\delta) \Bigr |
\nonumber
\\[-8pt]
\\[-8pt]
&&\quad\leq(c_n+L_n)
\varepsilon.\nonumber
\end{eqnarray}
Combining \eqref{pb3} and \eqref{pb4}, applied to $\delta$ and
$\delta_{\mu_n}$,
\begin{eqnarray*}
&&\max_{P\in\mathcal{P}^{(n)}} R_n(P,\delta_{\mu_m})
\\
&&\quad\leq\inf
_{\delta}\max_{P\in\mathcal{P}^{(n)}} R_n(P,\delta) +
(c_n+L_n)\varepsilon+ \varepsilon'.
\end{eqnarray*}
Since $\varepsilon,\varepsilon' > 0$ are arbitrary, the assertion
follows.
\end{pf*}
\end{appendix}

\section*{Acknowledgments}
Y. Ritov was supported by an ISF grant. P. J. Bickel was supported
by the National Science Foundation under Grant DMS-1160319.
A. C. Gamst was supported by the National Science Foundation under Grant DMS-0906808,
and by the Assistant Secretary for Energy Efficiency and Renewable Energy under Contact No. DE-AC02-05CH11231,
specifically the Materials Project is supported by Department
of Energy's Basic Energy Sciences program under Grant EDCBEE.


%

\end{document}